\documentclass[11pt]{article}

\usepackage{fancyhdr}
\usepackage{amsfonts}
\usepackage{amsmath}
\usepackage{amssymb}
\usepackage{amsthm}
\usepackage{tikz}
\usepackage{accents}

\usepackage{mathrsfs}
\usepackage{amscd}
\usepackage{amstext}
\usepackage{authblk}

\usepackage{graphics}
\usepackage{graphicx}

\newlength{\fixboxwidth}
\newcommand{\re}{\mathbb{R}}\newcommand{\N}{\mathbb{N}}
\newcommand{\zz}{\mathbb{Z}}\newcommand{\C}{\mathbb{C}}

\newcommand{\Z}{{\zz}^d}

\newcommand{\R}{{\re}^d}

\newcommand{\cs}{{\mathcal S}}

\newcommand{\cl}{{\mathcal L}}

\newcommand{\cf}{{\mathcal F}}
\newcommand{\cfi}{{\cf}^{-1}}

\newcommand{\supp}{{\rm supp \, }}

\newcommand{\mix}{{\rm mix}}
\newcommand{\gf}{\mathcal{F}}
\newcommand{\unif}{{\rm unif}}

\newcommand{\bproof}{\begin{proof}}
\newcommand{\eproof}{\end{proof}}

\newcommand{\be}{\begin{equation}}
\newcommand{\ee}{\end{equation}}
\newcommand{\beq}{\begin{eqnarray}}
\newcommand{\beqq}{\begin{eqnarray*}}
\newcommand{\eeq}{\end{eqnarray}}
\newcommand{\eeqq}{\end{eqnarray*}}

\numberwithin{equation}{section}

\newtheorem{theorem}{Theorem}[section]
\newtheorem{definition}[theorem]{Definition}
\newtheorem{corollary}[theorem]{Corollary}
\newtheorem{lemma}[theorem]{Lemma}
\newtheorem{proposition}[theorem]{Proposition}
\newtheorem{remark}[theorem]{Remark}

\begin{document}

\title{Pointwise Multipliers for Besov Spaces of Dominating Mixed Smoothness - II}

\author[a,b]{Van Kien Nguyen\thanks{E-mail: kien.nguyen@uni-jena.de,\ kiennv@utc.edu.vn}}
\author[a]{Winfried Sickel\thanks{E-mail: winfried.sickel@uni-jena.de}}
\affil[a]{Friedrich-Schiller-University Jena, Ernst-Abbe-Platz 2, 07737 Jena, Germany}
\affil[b]{University of Transport and Communications, Dong Da, Hanoi, Vietnam}


\date{\today}

\maketitle
\begin{abstract}
We continue our investigations on pointwise multipliers for Besov spaces of dominating mixed smoothness. 
This time we study the algebra property of the classes  $S^r_{p,q}B(\R)$ with respect to pointwise  multiplication. 
In addition if $p\leq q$,  we are able to describe  the space of all pointwise multipliers for  $S^r_{p,q}B(\R)$. 
\end{abstract}

{\em Key words:} Pointwise multipliers; algebras with respect to pointwise multiplication, 
Besov spaces of dominating mixed smoothness; characterization by differences; localization property. 


\section{Introduction}


The regularity concept related to  Besov spaces of dominating mixed smoothness 
are standard in {\em Approximation Theory} \cite{T93b}, {\em Numerical Analysis} \cite{BG}, \cite{SST08}  and 
{\em Information-Based Complexity} \cite{NoWo08}, \cite{NoWo10}, \cite{NoWo12}.
However, there is also some interest in {\em Learning Theory} in those classes, at least in $S^r_{2,2} B(\R)$, $r >0$, see \cite{sc}, \cite{lki}. 

Assertions on pointwise multipliers belong to the key problems in the modern theory of function spaces.
In our previous paper \cite{KS16} we investigated 
the set of all pointwise multipliers $M(S^r_{p,p}B(\R))$ for the classes $S^r_{p,p}B(\R)$. 
It turned out that under the natural restrictions $1\le p,q \le \infty$ and $r>1/p$ this set is given by $S^r_{p,p}B(\R)_\unif$.
This assertion, formally, is completely parallel to the isotropic case where we have 
$M(B^r_{p,p}(\R)) = B^r_{p,p}(\R)_\unif$ ($1\le p,q\le \infty$, $r>d/p$).
However, in reality the proof of the result in the dominating mixed case is much more involved 
than in the isotropic case.
In the present paper our aim consists in an extension of the above characterization to the situation $p \le q \le \infty$.
In \cite{SS}, \cite{KS16b} we have shown for the isotropic case the characterization 
$M(B^r_{p,q}(\R)) = B^r_{p,q}(\R)_\unif$ ($1\le p \le q\le \infty$, $r>d/p$).
It turns out that this extension  has a counterpart in the dominating mixed case as well; we shall prove below
\be\label{00}
M(S^r_{p,q}B(\R)) = S^r_{p,q}B(\R)_\unif \qquad \mbox{if}\quad  1\le p \le q\le \infty \quad \mbox{and}\quad r>1/p \, .
\ee
The extension from the isotropic case to the dominating mixed case is by no means straightforward.
To our own surprise the dominating mixed case is much more sophisticated.
The standard method in the isotropic situation, paramultiplication, seems to be not appropriate.
We shall deal with the characterization by differences of the underlying spaces, sometimes mixed with the Fourier analytic description. 

Let us mention that the restrictions in \eqref{00} are natural.
In cases either $q < p$ or $r <1/p $ the isotropic counterpart of the  identity in \eqref{00} is not longer true.
We refer to \cite{SS} and \cite{KS16b}. 
\\
The paper is organized as follows.
In Section \ref{def} we collect what we need about the classes $S^r_{p,q}B(\R)$ including some tools from Fourier analysis 
and few basic inequalities for differences. 
The next Section \ref{main} is devoted to the mutliplier problem.
First we shall describe there some basics about pointwise multipliers. After that 
we list our main results. Finally, in Section \ref{proofs},  we collect all proofs.


\subsection*{Notation}


As usual $\N$ denotes the natural numbers, $\N_0 :=\N\cup\{0\}$, 
$\zz$ denotes the integers, 
$\re$ the real numbers, 
and $\C$ the complex numbers. The letter $d\in \N, \ d>1,$ is always reserved for the underlying dimension in $\R, \Z$ etc. 
By $[d]$ we mean the set
$[d]:=\{1, \ldots ,d\}$. If $k=(k_1,  \ldots  , k_d)\in \N_0^d$, then we put
\beqq
|k|_1 := k_1 + \ldots \, + k_d\, \qquad\text{and}\qquad |k|_\infty:= \max_{j=1, \ldots ,\, d} \, k_j \,.
\eeqq
Further,
by $\langle x,y\rangle$ or $x\cdot y$  we mean the usual Euclidean inner product in $\R$. Let 
\beqq
x\diamond y := (x_1y_1, \ldots  , x_dy_d)\in \R\,.
\eeqq

If $X$ and $Y$ are two normed spaces, the norm
of an element $x$ in $X$ will be denoted by $\|x\,|\,X\|$. 
The symbol $X \hookrightarrow Y$ indicates that the
identity operator is continuous. For two sequences $a_n$ and $b_n$ we will write $a_n \lesssim b_n$ if there exists a
constant $c>0$ such that $a_n \leq c\,b_n$ for all $n$. We will write $a_n \asymp b_n$ if $a_n \lesssim b_n$ and $b_n
\lesssim a_n$. 

Let $\cs(\R)$ be the Schwartz space of all complex-valued rapidly decreasing infinitely differentiable  functions on $\R$. 
The topological dual, the class of tempered distributions, is denoted by $\cs'(\R)$ (equipped with the weak topology).
The Fourier transform on $\cs(\R)$ is given by 
\[
\cf \varphi (\xi) = (2\pi)^{-d/2} \int_{\R} \, e^{-ix \xi}\, \varphi (x)\, dx \, , \qquad \xi \in \R\, .
\]
The inverse transformation is denoted by $\cfi $.
We use both notations also for the transformations defined on $\cs'(\R)$\,.


\section{Besov spaces of dominating mixed smoothness}\label{def}


The history of Besov spaces has started in 1951 with a paper by Nikol'skij \cite{Ni51}.
Nikol'skij had investigated the spaces $B^s_{p,\infty}(\R)$ there. Later, his Ph.D-studies Besov \cite{Bes1}, \cite{Bes2}
introduced the classes $B^s_{p,q}(\R)$, $1\le p,q \le \infty$, $s>0$. 
The dominating mixed counterparts $S^r_{p,q}B(\R)$, $1\le p,q \le \infty$, $r>0$, have been introduced by 
Nikol'skij \cite{Ni63} ($q=\infty$), Amanov \cite{Am1} and Dzabrailov \cite{Dz1}, \cite{Dz2}.
The main new feature of these classes consists in the cross-norm property, see Remark \ref{blabla} below.
Besov spaces of dominating mixed smoothness  represent a quite different way to extend Besov spaces from $\re$ to $\R$, $d>1$.


\subsection{The definition and some basic properties}


We introduce the spaces by using the Fourier analytic approach.
Let $\varphi_0 \in C_0^{\infty}({\re})$ be a non-negative function such that 
 $\varphi_0  \equiv  1$ on $[-1,1]$ and $\supp\varphi_0 \subset [-\frac{3}{2},\frac{3}{2}]$. 
For $j\in \N$ we define
\beqq
         \varphi_j(\xi) = \varphi_0(2^{-j}\xi)-\varphi_0(2^{-j+1}\xi) ,\qquad\ \xi \in \re\, , 
\eeqq
and  
\be\label{ws-101}
\varphi_{k}(x) := \varphi_{k_1}(x_1)\cdot \ldots \cdot
         \varphi_{k_d}(x_d)\, , \qquad  x \in \R, \quad k\in \N_0^d\,. 
\ee        
This implies 
\[
\sum_{k\in \N_0^d} \varphi_k(x) =  1 \qquad  \text{for all}\ x\in \R\, , 
\]
and
\[
\supp \varphi_k  \subset  \Big\{ x \in \R: 2^{k_\ell-1}\le |x_\ell| \le 3 \, 2^{k_\ell-1}\, , \quad \ell =1, \ldots \, , d\Big\}\,, \qquad k \in \N^d\, .
\]
With other words, $(\varphi_k )_{k \in \N_0^d}$ is a smooth dyadic decomposition of unity of tensor product type.

\begin{definition}
Let $(\varphi_k)_{k\in \N_0^d}$ be the above system. Let  $1\le  p,q\leq \infty$ and $r \in \re$.  Then  $ S^{r}_{p,q}B(\re^d)$ is the
         collection of all tempered distributions $f \in \mathcal{S}'(\R)$
         such that
\beqq
          \|\, f \, |S^r_{p,q}B(\R)\|_{\varphi} :=
         \bigg(\sum\limits_{k\in \N_0^d} 2^{r|k|_1  q}\, \|\, \cfi[\varphi_{k} \cf f]
         |L_p(\re^d)\|^q\bigg)^{1/q} <\infty
\eeqq 
with the ususal modifications if $q=\infty$.
\end{definition}

Of course,  $S^r_{p,q}B(\R)$ are Banach spaces and they are independent from the chosen generator $\varphi_0$ of the smooth dyadic decomposition of unity 
$(\varphi_k)_{k \in \N_0^d}$ in the sense of equivalent norms.
For those basic facts we refer to the monographs \cite{Am} and \cite{ST}.

\begin{remark}\label{blabla} \rm 
(i) If $d=1$ we get $ S^{r}_{p,q} B(\re) = B^r_{p,q}(\re)$.  
\\
(ii) One of the most remarkable properties of Besov spaces of dominating mixed smoothness consists in the following. If $f_i \in B^r_{p,q}(\re)$, $ i=1,  \ldots  , d$,
 then its tensor product 
\[
 f(x) :=  (f_1 \otimes f_2 \otimes \, \ldots \, \otimes f_d ) (x) = \prod_{i=1}^d f_i (x_i)\, , \qquad x = (x_1, \, \ldots \, , x_d) \in \R\, ,
 \]
belongs to $S^r_{p,q}B(\R)$ and  
\[
 \| \, f \, | S^r_{p,q}B(\R)\| = \prod_{i=1}^d \|\, f_i \, |B^r_{p,q} (\re)\| \, .
\]
With other words, Besov spaces of dominating mixed smoothness have a cross-norm.
\end{remark}


\subsection{Besov spaces of dominating mixed smoothness and differences}


First  we recall the definition of (isotropic) Besov spaces. For a multivariate function $f:\R\to \C$,  $m  \in \N$, $h  \in \R$ and $x \in \R$ we put
 \[
 \Delta_{h}^{m} f(x):= \sum_{\ell =0}^{m} (-1)^{m -\ell} \, \binom{m}{\ell} \, f(x + \ell h )
 \]
and
\[
 \omega_m (f,t)_p := \sup_{|h|<t} \|\, \Delta_{h}^{m} f\, |L_p (\R)\|\, , \qquad t>0\, .
\]
Let $1 \le p,q\leq \infty$,  $r > 0$ and $m\in\N$ such that $m-1 \le r <m$. Then the (isotropic) Besov space 
$B^r_{p,q}(\R)$ is a collection of all $f\in L_p(\R)$ such that 
\[
\|f|B^r_{p,q}(\R)\|:= \|\, f\, |L_p(\R)\| + \bigg( \sum_{j=0}^{\infty} \, \big(2^{jr}\,  \omega_m (f, 2^{-j})_p\big)^q\bigg)^{1/q} < \infty. 
\]
We refer to the monographs \cite{Ni75} and \cite{Tr83}.

Now we turn to Besov spaces of dominating mixed smoothness. 
Let $j \in [d]= \{1,2, \ldots  , d\}$, $m \in \N$, $h \in \re$ and $x \in \R$. We put
\beqq
 \Delta_{h,j}^{m} f(x):= \sum_{\ell =0}^{m} (-1)^{m-\ell} \, \binom{m}{\ell} \, 
f(x_1, \ldots  , x_{j-1}, x_j + \ell h, x_{j+1}, \ldots  , x_d)\, .
\eeqq
This is the $m$-th order difference of $f$ in direction $j$.  
For $e\subset [d]$, $ h  \in \R$ and $m \in \N_0^d$ the mixed $(m,e)$-th difference operator $\Delta_h^{m,e}$ is defined to be 
 \[
 \Delta_{ h}^{m,e} := \prod_{i \in e} \Delta_{h_i,i}^{m_i} \qquad\mbox{and}\qquad \Delta_h^{m,\emptyset} :=  \operatorname{Id} \,, 
 \]
where $\operatorname{Id}f = f$. An associated modulus of smoothness is given by
\beqq
 \omega_{m}^e(f,t)_p:= \sup_{|h_i| < t_i, i \in  e}\|\, \Delta_h^{m,e}f \, |L_p(\R)\| \ ,\qquad t \in [0,1]^d\,,
\eeqq
where  $f \in L_p(\R)$ (in particular, $\omega_{m}^{\emptyset}(f,t)_p = \|f|L_p(\R)\| $). 
Many times,   e.g., in the Proposition below, we do not need to choose  $m$ as a vector.
For this reason, if $m\in \N$ we put 
$
\bar{m}:= (m, \ldots  , m)  \in \N_0^d 
$
and therefore
\[
\Delta_{ h}^{\bar{m},e} := \prod_{i \in e} \Delta_{h_i,i}^{m} \, .
\]
For a set $e\subset [d]$ we denote $e_0 :=[d] \backslash e$ and 
$$\N_0^d(e):=\big\{k\in \N_0^d: \ k_i=0\ \text{if}\ i\not \in e\big\}.$$
Let $k \in \N_0^d$.
For brevity we write $2^{-k} $ instead of the vector $(2^{-k_1}, 2^{-k_2}, \ldots  , 2^{-k_d})$.

\begin{proposition}\label{diff} 
Let $1 \le p,q\leq \infty$,  $r>0$ and $m\in \N$ such that $m-1 \le r < m$. 
Then the Besov space of dominating mixed smoothness $S^r_{p,q}B(\R)$
is the collection of all $f\in L_p(\R)$ such that
$$
    \|\, f \, | S^r_{p,q}B(\R)\|_{(m)} := 
\sum_{e\subset [d]}\bigg(\sum\limits_{k\in \N_0^d(e)} 2^{r|k|_1 q}\omega_{\bar{m}}^{e }(f,2^{-k})_p^{q}\bigg)^{1/q}\, 
$$
is finite (with the usual modification if $q=\infty$). 
Furthermore,  $\|\, \cdot \, | S^r_{p,q}B(\R)\|_{(m)}$ generates a norm equivalent to $\|\,\cdot \, | S^r_{p,q}B(\R)\|_{\varphi}$ on $L_p (\R)$.
\end{proposition}

This can be generalized as follows.

\begin{lemma}\label{diff1}
Let $1 \le p,q\leq \infty$ and $r>0$. 
Let $m \in \N_0^d$ such that $r < m_i$ for all $i\in [d]$.
Then
$$
   \|\, f \, | S^r_{p,q}B(\R)\|_{(m)} := 
\sum_{e\subset [d]}\bigg(\sum\limits_{k\in \N_0^d(e)} 2^{r|k|_1 q}\omega_{{m}}^{e }(f,2^{-k})_p^{q}\bigg)^{1/q}\, 
$$
is an equivalent norm on the space $S^r_{p,q} B(\R)$. 
\end{lemma}
For a proof  of both assertions we refer to \cite[2.3.4]{ST} ($d=2$) and \cite{U1}.
Sometime it is helpful to use the following characterization.
\begin{lemma}\label{red}
Let $ 1\le p,q \le \infty$ and $r>0$. Let $m \in \N$  such that $m>r$.
Then the Besov space of dominating mixed smoothness $S^r_{p,q}B(\R)$
is the collection of all $f\in L_p(\R)$ such that
\[
 \|\, f \, | S^r_{p,q}B(\R)\|_{(m)}^* := 
 \sum_{e\subset [d]} \Bigg\{\int\limits_{[-1,1]^{|e|}} \prod_{i \in e} |h_i|^{-rq} \big\|  \Delta_{  h}^{\bar{m},e} f(\cdot) \big|
L_p (\R)\big\|^q \prod_{i \in e} \frac{dh_i}{|h_i|}\Bigg\}^{1/q} 
 \]
is finite (with the usual modification if $q=\infty$). 
Furthermore,  $\|\, \cdot \, | S^r_{p,q}B(\R)\|_{(m)}^*$ generates a norm equivalent to $\|\,\cdot \, | S^r_{p,q}B(\R)\|_{\varphi} $ on $L_p (\R)$.
\end{lemma}
\begin{remark}\rm
 A proof of a slightly modified statement (integration with respect  to the components $t_i$ is taken on $(0,\infty)$, not on $(0,1]$)
can be found in \cite{U1}. The reduction to the case considered in Lemma \ref{red} can be done by standard arguments, we omit details. 
\end{remark}

Later on we shall need also the following embedding result. By $C(\R)$ we denote the collection of all uniformly continuous and 
bounded functions $f: ~ \R \to \C$, equipped with the sup-norm. 
 
\begin{lemma}\label{emb1}
Let $1\le  p,q\leq \infty$ and $r\in \re$. Then
the space $S^r_{p,q} B(\R)$ is continuously embedded into $C (\R)$ if and only if either $r>1/p$ or $r=1/p$ and $q= 1$.
\end{lemma}

For a proof we refer to \cite[2.4.1]{ST} ($d=2$),  \cite{Vybiral} and \cite{HaVy}. 

\begin{remark}
 \rm
 It is one of the remarkable observations that $S^r_{p,q} B(\R)$ many times  behaves like a Besov space defined on $\re$. 
\end{remark}


\subsection{Tools from Fourier analysis}


Next we will collect some required tools from Fourier analysis. We recall an adapted version of the famous Nikol’skij
inequality, see Uninskij \cite{Un1,Un2},  St\"ockert \cite{St} or \cite[Theorem 1.6.2]{ST}.

\begin{proposition}\label{Nikolski} 
Let $1 \le p_0\leq p\leq \infty$ and ${\alpha}=(\alpha_1, \ldots ,\alpha_d)\in \N_0^d$. 
Let $\Omega=[-b_1,b_1]\times \cdots \times [-b_d,b_d]$, $b_i>0$, $i=1, \ldots ,d$. 
Then there exists a positive constant $C$, independent of $(b_1, \ldots , b_d)$, such that 
\beqq
\| D^{\alpha}f|L_{p}(\R)\| \leq C\bigg(\prod_{i=1}^d b_i^{\alpha_i+\frac{1}{p_0}-\frac{1}{p}}\bigg) \|f|L_{p_0}(\R)\|
\eeqq
holds for all $f\in L_{p_0}(\R)  $ with $\supp \gf f \subset \Omega$\,.
\end{proposition}

The following construction of a maximal function is essentially  due to Peetre, but based on earlier work of Fefferman and Stein.
Let $a>0$ and  $b=(b_1, \ldots ,b_d)$, $b_i>0$, $i=1, \ldots ,d$ be fixed. Let $f$ be a regular distribution  such that  $\gf f$ is compactly
supported. We define the Peetre maximal function $P_{b,a}f$ by
\beqq
  P_{b,a}f(x) := \sup\limits_{z\in \R} \frac{|f(x-z)|}{\prod_{i=1}^d(1+|b_iz_i|)^a}\, , \qquad x \in \R\, .
\eeqq

\begin{proposition}\label{peetremax}
Let $1\le  p \leq\infty$ and $\Omega=[-b_1,b_1]\times \cdots \times [-b_d,b_d]$, $b_i>0$, $i=1, \ldots ,d$. Let further $a>1/p$. 
Then there exists a positive constant $C$,  independent of $(b_1, \ldots , b_d)$, such that
\beqq
\big\| P_{b,a}f \big|L_p(\R)\big\|\leq C\, \|f |L_p(\R)\|
\eeqq
holds for all $f \in L_p(\R)$ with $\supp (\gf f)\subset \Omega$.
\end{proposition}

For a proof we refer to \cite[Thm.~1.6.4]{ST}. 
A very useful  relation between Peetre maximal function and differences is given by the following 
lemma, see \cite{U1} and  \cite[2.3.3]{ST} (two-dimensional case).

 \begin{lemma}
 Let $a>0$ and $m \in \N$.  
 Then there exists a constant $C$
 such that 
\beqq
     |\Delta^m_hf(t)| \leq  C\, \max\{1,|bh|^a\}\, \min\{1,|bh|^m\}\, P_{b,a}f(t)\,.
\eeqq
 holds for all $b >0$, all $h\neq 0$, all $t\in \re$ and all $f\in \cs'(\re)$ satisfying $\supp(\gf f) \subset
 [-b,b]$.  
 \end{lemma}

 Applying the above result iteratively  with respect to components in $e\subset [d]$ we get the following modified 
 version in the multivariate situation.
 
 \begin{lemma}\label{ddim-1}
 Let  $a>0$, $e\subset [d]$, $m \in \N_0^d$  and $h =
 (h_1, \ldots ,h_d) \in \R$. Let further $f\in
 \mathcal{S}'(\R)$ with $\supp(\mathcal{F} f) \subset Q_{b}$, where
 $$
   Q_{b}:=[-b_1,b_1]\times \ldots \times [-b_d,b_d]\,,\ \ b_i>0,\ \ i=1, \ldots ,d.
 $$
 Then there exists a constant $C>0$ (independent of $f$, $b$, $x$
 and $h$) such that
 \begin{equation*}
 |\Delta^{{m},e}_h  f(x)|
 \leq C\, \bigg(\prod\limits_{i\in e}\, \max\big\{1,|b_ih_i|^a\big\}\, \min\big\{1,|b_ih_i|^{m_i}\big\}
 \bigg)\,  P_{b,a} f(x)
 \end{equation*}
 holds for all $x\in \R$. 
 \end{lemma}
 Let $m\in \N$. Then  $C^m_\mix (\R)$ is the collection of all continuous functions 
 $f:~\R\to \C $
 such that all  derivatives $D^\alpha f$ with $\max_{j=1, \ldots  ,d} \, \alpha_j\le m$ are continuous 
 and
$
  \sup_{|\alpha|_\infty\leq m} \sup_{x \in \R}\, |\, D^{\alpha}f(x)\, |< \infty \, .
$
 \begin{lemma}\label{ddim-2}  Let $b = (b_1, \ldots ,b_d)>0$, $a>0$, $e\subset [d]$, $m \in \N$, $\psi \in C^k_{\mix}(\R)$ with
 $k\geq m$ and $h =
 (h_1, \ldots ,h_d) \in \R$. Let further $f\in
 \mathcal{S}'(\R)$ with $\supp(\mathcal{F} f) \subset Q_{b}$, where
 $$
   Q_{b}:=[-b_1,b_1]\times \ldots \times [-b_d,b_d]\,.
 $$
 Then there exists a constant $C>0$ (independent of $f$, $b$
 and $h$) such that
 \begin{equation*}
 |\Delta^{\bar{m},e}_h (\psi\, \cdot \,  f)(x)|\\
 \leq C_{m,a,\psi}\bigg(\prod\limits_{i\in e}\max\big\{1,|b_ih_i|^a\big\}\min\big\{1,|b_ih_i|^m\big\}\bigg) P_{b,a} f(x)
 \end{equation*}
 holds for all $x\in \R$. 
 \end{lemma}
 
 \begin{remark}\rm
For a proof we refer to \cite{NUU}. Note that the constant $C_{m,a,\psi}$ depends on $m$, $a$, and $\sup_{ |\alpha|_\infty \leq k}\sup_{x\in \R}|D^{\alpha}\psi(x)|  $ only.
 \end{remark}


\section{Pointwise multipliers for Besov spaces of dominating mixed smoothness}
\label{main}



\subsection{Some generalities on pointwise multipliers}


For a quasi-Banach space $X$ of functions  we shall call a function $g$ a pointwise multiplier
if $g \, \cdot \, f \in X$ for all $f \in X$
(this is includes, of course, that the operation $f \mapsto g \, \cdot \, f$ must be well defined for all $f\in X$).
If $X \hookrightarrow L_p (\Omega)$ for some $p$ (here $\Omega$ is a domain in $\R$),  
as a consequence of the Closed Graph Theorem, we obtain that the liner operator 
$T_g : ~ f \mapsto g \, \cdot \, f$, associated to such a pointwise multiplier, must be continuous in $X$,
see \cite[p.~33]{MS2}. By $M(X)$ we denote the set of all pointwise multipliers for $X$, i.e.,
\[
M(X):= \big\{g:~ g  \cdot   f \in X \quad \forall f\in X\big\}
\]
and equip this set with the norm of the operator $T_g$
\[
\|\, g\, |M(X)\|:= \| \, T_g : ~ X \to X\| = \sup_{\|f|X\|\le 1}\, \| \, g \cdot f \, |X\|\, .
\]
We shall call $X$ an algebra with respect to pointwise multiplication 
(for short a multiplication algebra) if $f  \cdot  g \in X$ for all $f,g\in X$ and there exist a constant $C>0$ such that 
\beqq
\| f\cdot g\,|\, X\| \leq C\|\,f\,|\,X\| \cdot \|\, g\,|\,X\|
\eeqq
holds for all $f,g\in X$. It is obvious that if $X$ is a multiplication algebra we have, $X\hookrightarrow M(X)$.

\begin{lemma}\label{glatt}
Let $1\le p,q\le \infty$ and $r\in \re$. 
Then we have  $C_0^\infty (\R) \subset M(S^r_{p,q}B(\R))$.
\end{lemma}

Let $\psi$ be a non-negative $C_0^{\infty}(\R)$ function. We put $
\psi_{\mu}(x)=\psi(x-\mu)$, $\mu\in \Z,\ x\in \R
$
and assume that
\be\label{ws-10}
\sum_{\mu\in \Z} \psi_{\mu}(x)=1\qquad  \text{ for all}\ x\in \R\,.
\ee

\begin{definition}\label{def-unif}  
Let the Banach space $X$ be continuously embedded into $\cs' (\R)$.
Let  $\psi$ be as in \eqref{ws-10}. 
Then $X_{\unif}$ is the collection of all $f\in \cs' (\R)$ such that
\beqq
\|\, f\, | X_{\unif}\|_{\psi} := \sup_{\mu\in \Z} \, \|\, \psi_{\mu}\, \cdot \, f\, |  X\|<\infty.
\eeqq
\end{definition}

\begin{remark}\rm 
The spaces  $ S^r_{p,q} B(\R)_{\unif}$ are independent of the special choice of $\psi$ 
(in the sense of equivalent norms). This is an immediate  consequence of Lemma  \ref{glatt}.  
\end{remark}

\begin{lemma}\label{uniform} 
Let $1\le p,q\le \infty$ and $r\in \re$. 
Then the continuous embedding  
\[
M(S^r_{p,q}B(\R)) \hookrightarrow S^r_{p,q}B(\R)_\unif 
\]
takes place.
\end{lemma}


\subsection{Pointwise multipliers and algebras}


Our first  main  result  with respect to   Besov spaces of dominating mixed smoothness reads as follows.

\begin{theorem}\label{main-be}
Let $1 \le p,q\leq \infty$ and $r \in \re$. Then $S^r_{p,q}B(\R)$ is a multiplication algebra if and only if 
\begin{itemize}
\item either  $r>1/p$
\item or $1\leq p < \infty$, $r=1/p$ and $q = 1$.
\end{itemize}  
\end{theorem}

\begin{remark}\rm
There is a rich literature concerning this problem for the isotropic Besov spaces $B^s_{p,q} (\R)$.
We refer to Peetre \cite{Pe1}, Triebel \cite{Tr1}, \cite[2.6.2]{Tr78} and Mazya, Shaposnikova \cite{MS1}, \cite{MS2}. 
The little supplement, that $B^0_{\infty, q} (\R)$, $0< q \le 1$, is not an algebra, has been proved in  
\cite[4.6.4,~4.8.3]{RS}.
With respect to the dominating mixed Besov spaces we refer to \cite{KS16}, where sufficient conditions in case $p=q$ are treated.
\end{remark}

Our second main result consists in the description of the multiplier space under certain restrictions.

\begin{theorem}\label{mul-spaceb}
 Let $1  \le p \leq q \leq \infty$ and $r>1/p$. Then
\be\label{ws-02}
M(S^r_{p,q}B(\R))=S^r_{p,q}B(\R)_{\unif}
\ee
holds in the sense of equivalent norms.
\end{theorem}

\begin{remark}
\rm
{\rm (i)} 
In proving the characterization in \eqref{ws-02} we partly follow  the same strategy as in case of Theorem \ref{main-be}.
However, the proof is much more sophisticated than the proof of Theorem \ref{main-be}.
\\
{\rm (ii)} In case $p=q$ the result \eqref{ws-02} has been proved in \cite{KS16}.
\\
{\rm (iii)} The isotropic counterpart of Theorem \ref{mul-spaceb}, namely the identity
\beqq
M(B^s_{p,q}(\R))=B^s_{p,q}(\R)_{\unif}\, , \qquad 1  \le p \leq q \leq \infty, \quad s>d/p, 
\eeqq
has been known for some years in the special case $p=q$, we refer to  
Strichartz \cite{Str-67} ($p=q=2$), Peetre \cite{Pe}, page 151, ($1 \le p=q \le \infty $), Maz'ya and Shaposnikova, see 
\cite[Theorems 4.1.1, 5.3.1, 5.3.2, 5.4.1]{MS2}, ($1 \le  p=q < \infty $).
S. \cite{Si} ($1\le p=q < \infty$) and    Triebel \cite[Proposition 2.22]{Tr06}.
The case $p < q$ has been proved  for the first time   in S. and Smirnov \cite{SS}. 
Quite recently a different proof has been  given by  the authors \cite{KS16b}.
\end{remark}

By using duality arguments one can derive from Theorem \ref{mul-spaceb} the following.

\begin{corollary}\label{mul-spacec}
 Let $1   < q \le p< \infty$ and $r< \frac 1p -1$. Then
\be\label{ws-127}
M(S^r_{p,q}B(\R))=S^{-r}_{p',q'}B(\R)_{\unif}
\ee
holds in the sense of equivalent norms.
\end{corollary}

In the isotropic case it is well-known that 
Theorem \ref{main-be} can be improved in the following way.
Let $1 \le p,q\leq \infty$ and $s >0$. Then $B^s_{p,q}(\R) \cap L_\infty (\R)$ is a multiplication algebra and 
there exists a constant $C$ such that  
\[
\|\, f\, \cdot \, g\, | B^s_{p,q}(\R)\| \leq C\, \big(\|\, f\, |B^s_{p,q}(\R)\| \, \|\, g\, |L_{\infty}(\R)\| + 
\|\, f\, |L_{\infty}(\R)\|\, \|\, g\, |B^s_{p,q}(\R)\|\big)
\]
holds for all $f,g\in B^s_{p,q}(\R)$.
Inequalities of this type are sometimes called Moser inequalities.
In the dominating mixed case those Moser-type inequalities are not true.

\begin{theorem}\label{negative}
Let $d>1$, $1\le  p,q\leq \infty$ and $r>0$. Then there exists no constant $C>0$ such that
\beqq
\|\, f\cdot g\, | S^r_{p,q}B(\R)\| \leq C\, \big(\|\, f\, |S^r_{p,q}B(\R)\|\, \|\, g\, |L_{\infty}(\R)\| + 
\|\, f\, |L_{\infty}(\R)\|\, \|\, g\, |S^r_{p,q}B(\R)\|\big)
\eeqq
holds for all $f,g\in S^r_{p,q}B(\R) \cap L_\infty(\R)$.
\end{theorem}


\subsection{Pointwise multipliers and algebras - the local case}


As a service for the reader we investigate the local situation as well, i.e., we consider 
 pointwise multipliers  for Besov spaces  of dominating mixed smoothness defined on the cube   $\Omega:= [0,1]^d$. 
For convenience  we introduce the  spaces under consideration by taking restrictions.

\begin{definition} 
Let $1\le  p,q\leq \infty$ and $r \in \re$. Then
      $S^{r}_{p,q}B(\Omega)$ is the space of all $f\in D'(\Omega)$ such that there exists   $g\in
      S^{r}_{p, q }B(\R)$ satisfying $f = g|_{\Omega}$. It is endowed with the quotient norm
      $$
         \|\, f \, |S^{r}_{p,q}B(\Omega)\| = \inf \Big\{ \|g|S^{r}_{p,q}B(\R)\|~:~ g|_{\Omega} =
         f \Big\}\,.
      $$
\end{definition}

Our main results as listed in the previous subsection carry over to the local case.

\begin{theorem}\label{main-be-1}
Let $1\le  p,q\leq \infty$ and $r\in \re$. Then $S^r_{p,q}B(\Omega)$ is a multiplication algebra if and only if 
\begin{itemize}
\item either  $r>1/p$
\item or $1\leq p < \infty$, $r=1/p$ and $q = 1$.
\end{itemize} 
\end{theorem}

In the local case  Theorem \ref{main-be-1} can be immediately turned into a satisfactory  characterization of 
 $M(S^r_{p,q}B(\Omega))$.

\begin{theorem}\label{mul-spacew}
Let $1\le  p,q\leq \infty$ and $r>1/p$. Then
\[
M(S^r_{p,q}B(\Omega))=S^r_{p,q}B(\Omega)
\]
holds in the sense of equivalent norms.
\end{theorem}

Also in the local situation a Moser-type inequality does not hold.

\begin{theorem}\label{negativec}
Let $d>1$, 
$1\le p,q\leq \infty$ and $r>0$. Then there exists no constant $C>0$ such that
\beqq
\|f\, \cdot \, g| S^r_{p,q}B(\Omega)\| \leq C\, \big(\|\,f\,|S^r_{p,q}B(\Omega)\|\, \|\, g\, |L_{\infty}(\Omega)\| 
+ \|\, f\, |L_{\infty}(\Omega)\|\, \|\, g\, |S^r_{p,q}B(\Omega)\|\big)
\eeqq
holds for all $f,g\in S^r_{p,q}B(\Omega) \cap L_\infty (\Omega)$.
\end{theorem}


\section{Proofs}\label{proofs}


All proofs are collected in this section.
We postpone the proof of Lemma \ref{glatt} and Lemma \ref{uniform} to the Subsection \ref{proofs-Teil1}.


\subsection{Proof of the algebra property}


\vskip 0.3cm
\noindent
{\bf Proof of Theorem \ref{main-be}}. {\it Step 1. } Let $r <m\leq r+1$. Since the norm  $\|\cdot\,|\,S^r_{p,q}B(\R)\|_{(m)}$ 
does not depend on $m>r$ in the sense of equivalent norms, see Lemma \ref{diff1}, we shall prove that  
\beqq
     \|\, f\, \cdot \, g\,|\,S^r_{p,q} B(\R)\|_{(2m)} \leq C \,  \| \, f\, |S^r_{p,q}B(\R) \| \,  \|\,  g \, |S^r_{p,q}B(\R) \|
\eeqq 
holds for all $f,g\in S^r_{p,q}B(\R)$. Taking into account Lemma \ref{emb1} we obtain
\beqq
\| f\, \cdot \, g\,|L_p(\R)\| \, \leq\, \| f|L_p(\R)\|\cdot \| g|C(\R)\|\, \leq\,c\, \| f|S^r_{p,q}B(\R)\|\cdot \|g|S^r_{p,q}B(\R)\|.
\eeqq  
This inequality should  be interpreted as the estimate needed for the term with $e= \emptyset$.
Next we need some identities for differences.
Note that if $\psi,\, \phi:\ \re \to \C$ and $m \in \N$ we have
\be\label{formular}
\Delta_h^{m}(\psi \phi)(x)=\sum_{j=0}^{m} \binom{m}{j} \, \Delta_h^{m-j}\psi(x+j h)\, \Delta_h^{j}\phi(x),\qquad x,h\in \re \, ,
\ee 
which can be proved by induction on $m$.
Let $e\subset [d]$, $e\not=\emptyset$ and recall the notation $e_0=[d]\backslash e$,
\[
x \diamond y = (x_1\, \cdot \,y_1, \ldots \,, x_d\, \cdot   y_d)\in \R\, 
\]
and
\[
\N_0^d(e) = \big\{k\in \N_0^d: \ k_i=0\ \text{if}\ i\not \in e\big\}.
\]
Then we derive from \eqref{formular} that
\be\label{mot}
\Delta_{h}^{2{\bar{m}},e}(f\, \cdot \,  g)(x)=\sum_{ u \in \N_0^d(e),\,|u|_\infty\leq 2m} \binom{2\bar{m}}{u}\, 
\Delta_{h}^{2 \bar{m} - u ,e}f(x+ u  \diamond h)\, \Delta_{h}^{ u ,e}g(x)\,,\quad\, x, h\in \R\, , 
\ee 
holds. Here $2\bar{m}-u :=(2m-u_1, \ldots ,2m-u_d)$ and 
\[
\binom{2\bar{m}}{u} = \prod_{i \in e} \binom{2m}{u_i} \, .
\]
The main step of the proof will consist  in the  estimates of the terms
\be\label{ws-117}
S_{e,u} :=\Bigg\{\sum\limits_{k \in \N_0^d(e)} 2^{r|k|_1 q}\bigg(\sup_{|h_i| < 2^{-k_i}, i \in e} 
\big\| \Delta_{h}^{2 \bar{m} - u ,e}f(\cdot+ u  \diamond h)\Delta_{h}^{ u ,e}g(\cdot)|L_p(\R)\big\|\bigg)^q\Bigg\}^{1/q}
\ee 
$e \not= \emptyset$, $u \in \N_0^d(e)$, $|u|_\infty \le 2m$.
Therefore we have to  consider different cases.
\\
{\it Step 2.} The case $u_i \le m$ for all $i\in e$. Obviously we have  $2m-u_i \ge m $ for all $i \in e$. 
Using a change of variables in the $L_p$-integral in the second step we obtain for a certain constant $c_1$
\beqq
 \big\| \Delta_{h}^{2\bar{m} -u ,e}f(\cdot+u  \diamond h)\Delta_{h}^{u ,e}g(\cdot)\big|L_p(\R)\big\|  
 &\le &   \big\| \Delta_{h}^{2\bar{m} -u ,e}f(\cdot+u  \diamond h)\big|L_p(\R)\big\|\,  \sup_{x\in \R} |\Delta_{h}^{u ,e}g(x)|\\
& \le &  c_1 \, \| g |C(\R)\|\,  \big\| \Delta_{h}^{\bar{m} ,e}f(\cdot )\big|L_p(\R)\big\|\,. 
\eeqq
The embedding $ S^r_{p,q}B(\R) \hookrightarrow C(\R)$, see Lemma \ref{emb1},    implies 
\beqq
\sup_{ |h_i| < 2^{-k_i}, i \in e} \big\| \Delta_{h}^{2\bar{m} -u ,e} f(\cdot+u  \diamond h)\Delta_{h}^{u ,e} \, g(\cdot)\, \big|L_p(\R)\big\| 
& \le &   c_1\,  \big\|g |C(\R)\big\| \,  \omega_{\bar{m}}^{e}(f,2^{-k })_p   
\\
&\le &  c_2\, \big\|g |S^r_{p,q}B(\R)\big\|\,   \omega_{\bar{m}}^{e}(f,2^{-k })_p  
\eeqq
with an appropriate constant $c_2$. Consequently we have 
\beqq
S_{e,u}
& \le &  c_2 \, \big\|g \big|S^r_{p,q}B(\R)\big\|  \, \bigg(\sum\limits_{k  \in \N_0^d(e)} 2^{r|k |_1 q}\, \omega_{\bar{m}}^{e}(f,2^{-k })^{q}_p\bigg)^{1/q}   
\\
& \le & c_2\, \big\|g \big|S^r_{p,q}B(\R)\big\|\,  \big\|f \big|S^r_{p,q}B(\R)\big\|\, .
\eeqq
The case $u_i\geq m$ for all $ i\in e$ can be handled in the same way by interchanging the roles of $f$ and $g$.
\\
{\it Step 3.} The remaining cases.  Without loss of generality we may assume that $e=\{1, \ldots ,N\}$ for some natural number $N$, $N \le d$. 
In addition we assume 
\[
u=(u_{1}, \ldots ,u_L,u_{L+1}, \ldots ,u_N,0,\ldots,0)
\] 
with
$$ m\leq u_i\leq 2m,\quad i=1, \ldots ,L,\qquad 0\leq u_i<m, \quad i= L+1, \ldots ,N$$
and $1\leq L\le N$ and $L< d$. 
For brevity we put
\[
e_1:= \{L+1,\ldots,N\} \qquad \mbox{ and } \qquad  e_2:=\{1,\ldots,L \}\, . 
\]
By assumption both sets are nontrivial.
This covers all remaining cases up to an enumeration.
\\
{\it Substep 3.1.} Let $1\leq p,q\leq \infty$ and $r>1/p$. Obviously it holds  $\N_0^d(e)=\N_0^d(e_1)\cup \N_0^d(e_2)$. 
Any  $k\in \N_0^d(e)$ can be written as $k=k^1+k^2$ with $k^1\in \N_0^d(e_1)$ and $k^2\in \N_0^d(e_2)$ in an unique way. 
Next we apply  the tensor product system $(\varphi_{j})_{j\in \N_0^d}$, defined in \eqref{ws-101}.  
We shall use the convention that in the univariate case $\varphi_n \equiv 0$ if $n< 0$, which implies that 
$\varphi_{(j_1, \ldots  , j_d)} \equiv 0$ if $\min_i j_i < 0$. For any $k \in \N_0^d$
this yields 
\[
 f(x) = \sum_{ \ell \in \zz^{d}}  \gf^{-1} [ \varphi_{k+ \ell } \gf f](x)
\]
with convergence in $ S^r_{p,q}B(\re^d)$ and therefore in $C(\re^d)$, see Lemma \ref{emb1}. 
In particular, we have the decompositions
\beqq
f(x) = \sum_{{\ell} \in \zz^d}  \gf^{-1}  [\varphi_{k^1 + \ell } \gf f](x)  \qquad 
\text{and}\qquad 
g(x)=\sum_{ {\nu} \in \zz^d}  \gf^{-1}  [\varphi_{k^2 + \nu}  \gf g](x)\, , \qquad x \in \R\, , 
\eeqq 
with convergence in $C(\R)$. To simplify notation we put
\[
f_{{\ell} } : =\gf^{-1} [\varphi_{ {\ell} } \gf f] \qquad \mbox{and} \qquad g_{{\ell} } : =\gf^{-1}  [\varphi_{{\ell} } \gf g]\, , \qquad 
\ell \in \zz^d\, .
\]  
An application of the triangle inequality leads to
 \beqq
 \big\| \Delta_{h}^{2 \bar{m} - u ,e}f(\cdot+ u  \diamond h)&& \hspace{-0.7cm} 
 \Delta_{h}^{ u ,e}g(\cdot)\,  |L_p(\R)\big\|
\nonumber
\\
&\le &  \sum_{  \ell , \nu \in \zz^d}  \big\| \Delta_{h}^{2 \bar{m} - u ,e}f_{k^1 +\ell }(\cdot+ u  \diamond h)
\Delta_{h}^{ u ,e}g_{k^2+\nu }(\cdot  ) \, |L_p(\R)\big\|\,.
\eeqq
We will estimate the sum on the right-hand side term by term.
It follows
\beqq
&& \hspace{-0.7cm}\big\| \Delta_{h}^{2 \bar{m} - u ,e}f_{k^1 +\ell }(\cdot+ u  \diamond h)
 \Delta_{h}^{ u ,e}g_{k^2+\nu }(\cdot  )\,  |L_p(\R)\big\|
\\
& \le &  \bigg(\int\limits_{\re^{d-L}} \sup_{\substack{x_i\in \re\\i \le L}} 
\big|\Delta_{h}^{2 \bar{m} - u ,e}f_{k^1 +\ell }(x+ u  \diamond h) \big|^p \prod_{i=L+1}^d  d  x_i \bigg)^{1/p} 
\bigg(\int\limits_{\re^{L}} \sup_{\substack{x_i\in \re\\i > L}} \big|\Delta_{h}^{ u ,e }g_{k^2+\nu }(x) \big|^p
\prod_{i=1}^L  dx_i \bigg)^{1/p}.
\eeqq
Let $\gf_L$ denote the Fourier transform with respect to $(x_1, \ldots  , x_L)$.
Observe, that for any $h \in \re^L$
\[
\supp  \gf_L  \big(f_{k^1+\ell} (\, \cdot \, + h, x_{L+1} , \ldots  , x_d)\big)\,  
\subset \big\{(\xi_1, \ldots  , \xi_L):~ |\xi_j|\le 3 \, 2^{k_j^1+\ell_j-1}\, , \: j=1, \ldots   , L \big\}  ,
\]
independent of $x_{L+1} , \ldots \, ,x_d$. Consequently, Nikol'skijs inequality in Proposition \ref{Nikolski} yields
\beqq
 \bigg(\int\limits_{\re^{d-L}} \sup_{\substack{x_i\in \re\\i \le L}} && \hspace{-0.7cm} 
\big|\Delta_{h}^{2 \bar{m} - u ,e}f_{k^1 +\ell }(x+ u  \diamond h) \big|^p \prod_{i=L+1}^d  d  x_i \bigg)^{1/p} 
\\
& \le  &  c_3\,  \prod_{i=1}^L 2^{\frac{k_i^1+\ell_i}{p}} 
\bigg(\int\limits_{\re^{d}}  
\big|\Delta_{h}^{2 \bar{m} - u ,e} f_{k^1 +\ell }(x+ u  \diamond h) \big|^p  dx \bigg)^{1/p} 
\\
& \le  &  c_4\,  \prod_{i \in e_2} 2^{\frac{\ell_i}{p}} 
\bigg(\int\limits_{\re^{d}}  
\big|\Delta_{h}^{ \bar{m}   ,e_1} f_{k^1 +\ell }(x) \big|^p  dx \bigg)^{1/p}
\eeqq
with  constants $c_3, c_4$ independent of $f$, $k$ and $\ell$, since $k^1\in \N_0^d(e_1)$, i.e., $k_i^1=0$ if $i\in e_0\cup e_2$. 
A simple change of coordinates  and an analogous argument with 
respect to $g_{k^2 + \nu}$ results in 
\beqq
&& \hspace{-0.7cm} 
\big\| \Delta_{h}^{2 \bar{m} - u ,e}f_{k^1 +\ell }(\cdot+ u  \diamond h)
\Delta_{h}^{ u ,e}g_{k^2+\nu }(\cdot  )\,  |L_p(\R)\big\|  
\\
& \le & c_5 \, \Big( \prod_{i\in e_2} 2^{\frac{\ell_i}{p}}\Big) \Big(\prod_{ i\in (e_0\cup e_1)} 2^{\frac{\nu_i}{p}}\Big) 
\big\|\Delta_{h}^{\bar{m} ,e_1}f_{k^1 +\ell } \big|L_p (\R)\big\|  \, \big\|\Delta_{h}^{\bar{m} ,e_2}g_{k^2+\nu} \big|L_p (\R)\big\|  
\, .
\eeqq
We need one more notation. For $\ell \in \Z$ we put
\[
 \omega (\ell):= \big\{i \in \{1, \ldots , d\}:~ \ell_i <0 \big\} \qquad \mbox{and}\qquad 
 \overline{\omega} (\ell):= \big\{i \in \{1, \ldots , d\}:~ \ell_i \ge 0 \big\}\, .
\]
Since $k^1\in \N_0^d(e_1)$  and $\varphi_{k_i^1+\ell_i}\equiv 0 $ if $k_i^1+\ell_i<0$,  we can assume that 
\be\label{ell}
(e_0\cup e_2)\subset \overline{\omega}(\ell)\qquad \text{and therefore}\qquad \omega(\ell)\subset e_1\, ;
\ee 
similarly  
\be\label{nu}
(e_0\cup e_1)\subset \overline{\omega}(\nu)\qquad\text{and}\qquad\omega(\nu)\subset e_2.
\ee 
Writing $\Delta_{h}^{ \bar{m}  ,e_1}$ as 
\[
\Delta_{h}^{  \bar{m}  ,e_1} = \Big(\prod_{i \in \overline{\omega} (\ell) \cap e_1} 
 \Delta_{h_i}^{  m }\Big) \Big(\prod_{i \in \omega (\ell)   } \Delta_{h_i}^{  m  }\Big)\, ,
\]
taking Lemma \ref{ddim-1} and Proposition \ref{peetremax} into account, 
it is easily seen that 
\beqq
 \sup_{|h_i|< 2^{-k_i} ,\, i\in  e} \, \big\|\Delta_{h}^{  \bar{m}   ,e_1}f_{k^1 +\ell } \big|L_p (\R)\big\|
 \le c_6 \, \prod_{i \in \omega (\ell)  } 2^{\ell_i m}\, \big\|f_{k^1 +\ell } \big|L_p (\R)\big\|\, , 
\eeqq
where we  used the second part in \eqref{ell} and the definition of $e_1$ as well. 
Similarly 
\[
 \sup_{|h_i|< 2^{-k_i} ,\, i\in  e} \, \big\|\Delta_{h}^{  \bar{m}   ,e_2}g_{k^2 +\nu } \big|L_p (\R)\big\|
 \le c_6 \, \prod_{i \in \omega (\nu)  } 2^{\nu_i m}\, \big\|\, g_{k^2 +\nu } \, \big|L_p (\R)\big\|\, .
\]
Altogether we have found the estimate
\beq \label{k-02}
&& \hspace{-0.7cm} \sup_{|h_i|< 2^{-k_i} ,\, i\in  e}  
\big\| \Delta_{h}^{2 \bar{m} - u ,e}f_{k^1 +\ell }(\cdot+ u  \diamond h)
\Delta_{h}^{ u ,e}g_{k^2+\nu }(\cdot  )\,  |L_p(\R)\big\|
\nonumber
\\
& \le &   c_7 \, \underbrace{\bigg( \prod_{i\in e_2} 2^{\frac{\ell_i}{p}} \prod_{ i\in (e_0\cup e_1)} 2^{\frac{ \nu_i}{p}} \prod_{i \in \omega (\ell) } 2^{\ell_i  m }  \prod_{i \in \omega (\nu)} 2^{\nu_i m}\bigg)
  \|f_{k^1 +\ell } |L_p (\R)\|  \, \|g_{k^2+\nu} |L_p (\R)\|}  .
\\
&& \qquad \qquad\qquad\qquad\qquad\qquad\qquad  P(f,g,k^1,k^2, \ell,\nu)
\nonumber
\eeq 
For simplicity we denote by $P(f,g,k^1,k^2, \ell,\nu)$ the term on the right-hand side in \eqref{k-02}. Hence, by applying triangle inequality we get
 \beq \label{k-08}
 S_{e,u}  & \leq & c_7\, \Bigg\{\sum\limits_{k \in \N_0^d(e)} 2^{r|k|_1 q}\Bigg[\sum_{\ell,\nu\in \Z} P(f,g,k^1,k^2,\ell,\nu) \Bigg]^q\Bigg\}^{1/q} 
 \nonumber
 \\
& \leq & 
 c_7\, \sum_{\ell,\nu\in \Z}\Bigg\{\sum\limits_{k \in \N_0^d(e)}   2^{r|k|_1 q}    P(f,g,k^1,k^2,\ell,\nu)^q  \Bigg\}^{1/q}.
\eeq 
Observe that
 \beq\label{ws-extra1}
 \sum_{k\in \N_0^d(e)} && \hspace*{-0.7cm} 2^{r|k^1+\ell|_1q}2^{r|k^2+\nu|_1q}\,\|\, f_{k^1 +\ell } \, |L_p (\R)\|^q  \, \|\, g_{k^2+\nu}\, |L_p (\R)\|^q 
\nonumber 
\\
&= & \bigg( \sum_{k^1\in \N_0^d(e_1)} 2^{r|k^1+\ell|_1q} \|\, f_{k^1 +\ell} \, |L_p (\R)\|^q \bigg) 
\bigg( \sum_{k^2\in \N_0^d(e_2)} 2^{r|k^2+\nu|_1q} \|\, g_{k^2+\nu}\, |L_p (\R)\|^q\bigg) 
\nonumber
\\
&\leq & \| \, f \, |S^r_{p,q}B(\R)\|^q\, \|\, g\, |S^r_{p,q}B(\R)\|^q\, .
 \eeq
Recall, we only need to consider those terms where $\min_i (k^1_i + \ell_i)\ge 0$ and 
$\min_i (k^2_i + \nu_i)\ge 0$. Hence we get for any $k \in \N_0^d(e)$, see \eqref{ell} and \eqref{nu}, 
\beqq 
&& \hspace{-0.7cm}
\sum_{i=1}^d |k_i| - \sum_{i=1}^d |k_i^1+\ell_i| - \sum_{i=1}^d |k_i^2+\nu_i| =  
\Big(\sum_{i=1}^L k_i^2  - \sum_{i=1}^L  \ell_i
-\sum_{i=1}^L (k_i^2 + \nu_i)\Big)
\\
& + & \Big(
\sum_{i=L+1}^N k_i^1  - \sum_{i=L+1}^N (k_i^1 + \ell_i)  - \sum_{i=L+1}^N  \nu_i\Big)
 -  \Big(\sum_{i=N+1}^d \ell_i + \sum_{i=N+1}^d \nu_i \Big)
\\
&=& - \sum_{i=1}^d (\ell_i + \nu_i) \, .
\eeqq

Again in view of \eqref{ell} and \eqref{nu}, this implies
\beq\label{ws-extra}
&& \hspace*{-0.8cm}\bigg(2^{r|k|_1 } \prod_{i\in e_2} 2^{\frac{ \ell_i}{p}} \prod_{ i\in (e_0\cup e_1)} 2^{\frac{  \nu_i}{p}} 
\prod_{i \in \omega (\ell) } 2^{\ell_i  m }   \prod_{i \in \omega (\nu)} 2^{\nu_i m}\bigg)\bigg(2^{-r|k^1+\ell|_1 }2^{-r|k^2+\nu|_1 } \bigg)
\nonumber
\\
& =  &\bigg( \prod_{i\in e_2} 2^{\frac{ \ell_i}{p}} \prod_{i \in \omega (\ell) } 2^{\ell_i  m } \prod_{i=1}^d2^{-r\ell_i}  \bigg)
\bigg(\prod_{i\in (e_0\cup e_1)} 2^{\frac{  \nu_i}{p}} \prod_{i \in \omega (\nu)} 2^{\nu_i m}\prod_{i=1}^d2^{-r\nu_i} \bigg) 
\nonumber
\\
& = & \bigg( \prod_{i\in e_2} 2^{\ell_i(\frac{1}{p}-r)} \prod_{i \in \omega (\ell) } 2^{\ell_i ( m-r) } 
\prod_{\overline{\omega}(\ell)\backslash e_2}2^{-r\ell_i}  \bigg)
\bigg(\prod_{i\in (e_0\cup e_1)} 2^{\nu_i(\frac{   1}{p}-r)} 
\prod_{i \in \omega (\nu)} 2^{\nu_i (m-r)}\prod_{i\in \overline{\omega}(\nu)\backslash (e_0\cup e_1)} 2^{-r\nu_i} \bigg)
\nonumber
\\
& \leq & \bigg(\prod_{i=1}^d 2^{-|\ell_i|\delta}\bigg)\bigg(\prod_{i=1}^d2^{-|\nu_i|\delta}\bigg)\, 
\eeq
where $\delta :=\min(r-1/p,m-r,r)>0$. Consequently we conclude that
\beqq
 S_{e,u}   & \leq &  c_7\,  \sum_{\ell,\nu\in \Z}\bigg(\prod_{i=1}^d 2^{-|\ell_i|\delta}\bigg)
\bigg(\prod_{i=1}^d2^{-|\nu_i|\delta}\bigg) \, \|\,  f\, |S^r_{p,q}B(\R)\|\, \|\, g\, |S^r_{p,q}B(\R)\| 
\\
&\leq & c_8\,   \| \, f\, |S^r_{p,q}B(\R)\|\,  \|\, g\, |S^r_{p,q}B(\R)\|
\eeqq
for an appropriate constant $c_8$ independent of $f$ and $g$.
\\
{\it Step 3.2.}
The case $1\le p <\infty$, $r=1/p$ and  $ q=1$. Our point of departure is the first inequality in \eqref{k-08}.
This yields
\beqq
 S_{e,u}  \leq c_7 \, \sum\limits_{k \in \N_0^d(e)} 2^{r|k|_1} \sum_{\ell,\nu\in \Z}\, P(f,g,k^1,k^2, \ell,\nu).
\eeqq
To continue we need another splitting of the summation as used in Substep 3.1.
We observe that
 \beqq
&& \hspace*{-0.7cm} \sum_{\ell_i\in \zz\atop i\in [d]\backslash e_1}
\sum_{\nu_i\in \zz\atop i\in [d]\backslash e_2} \sum_{k\in \N_0^d(e)} 2^{r|k^1+\ell|_1}\, 2^{r|k^2+\nu|_1}\,
\|f_{k^1 +\ell } |L_p (\R)\|  \, \|g_{k^2+\nu} |L_p (\R)\| 
\\
&=& \bigg( \sum_{\ell_i\in \zz\atop i\in [d]\backslash e_1}\sum_{k^1\in \N_0^d(e_1)} 2^{r|k^1+\ell|_1} \, \|f_{k^1 +\ell } |L_p (\R)\| \bigg)
 \bigg(\sum_{\nu_i\in \zz\atop i\in [d]\backslash e_2} \sum_{k^2\in \N_0^d(e_2)} 2^{r|k^2+\nu|_1} \|g_{k^2+\nu} |L_p (\R)\|\bigg) 
\\
&\leq & \| \, f\, |S^r_{p,q}B(\R)\|\,  \|\, g\, |S^r_{p,q}B(\R)\|
\eeqq
(as a replacement of \eqref{ws-extra1}) and
\beqq
\bigg(2^{r|k|_1 } && \hspace{-0.7cm}\prod_{i\in e_2} 2^{\frac{ \ell_i}{p}} \prod_{ i\in (e_0\cup e_1)} 2^{\frac{  \nu_i}{p}} 
\prod_{i \in \omega (\ell) } 2^{\ell_i  m }   \prod_{i \in \omega (\nu)} 2^{\nu_i m}\bigg)\bigg(2^{-r|k^1+\ell|_1 }2^{-r|k^2+\nu|_1 } \bigg)
\\
&= & \bigg( \prod_{i\in e_2} 2^{\frac{ \ell_i}{p}} \prod_{i \in \omega (\ell) } 2^{\ell_i  m } \prod_{i=1}^d2^{-r\ell_i}  \bigg)
\bigg(\prod_{i\in (e_0\cup e_1)} 2^{\frac{  \nu_i}{p}} \prod_{i \in \omega (\nu)} 2^{\nu_i m}\prod_{i=1}^d2^{-r\nu_i} \bigg) 
\\
& = & \bigg(  \prod_{i \in \omega (\ell) } 2^{\ell_i ( m-r) } \prod_{\overline{\omega}(\ell)\backslash e_2}2^{-r\ell_i}  \bigg)
\bigg( \prod_{i \in \omega (\nu)} 2^{\nu_i (m-r)}\prod_{i\in \overline{\omega}(\nu)\backslash (e_0\cup e_1)} 2^{-r\nu_i} \bigg)
\\
& \leq & \bigg(\prod_{i\in e_1} 2^{-|\ell_i|\delta_1}\bigg)\bigg(\prod_{i\in e_2}2^{-|\nu_i|\delta_1}\bigg)
\eeqq
with $\delta_1:=\min(m-r,r)$ (as a replacement of \eqref{ws-extra}).  
Now we can conlude as above that
\beqq
S_{u,e} \leq c_9\,  \| \, f\, |S^r_{p,q}B(\R)\|\,  \|\, g\, |S^r_{p,q}B(\R)\|
\eeqq
holds as well in this case. 
\\
{\it Step 4.} Necessity.
We shall work with tensor products of functions and the cross-norm property, see Remark \ref{blabla}.
Let us assume that $S^r_{p,q}B(\R)$ is an algebra with respect to pointwise multiplication.
Then all products of the form 
\[
 \Big(f(x_1) \,\cdot \,  \prod_{i=2}^d \psi (x_i)\Big) \cdot \Big(g(x_1) \,\cdot \,  \prod_{i=2}^d \psi (x_i)\Big)
\]
with $f,g \in B^r_{p,q}(\re)$ and $\psi \in C_0^\infty (\re)$ have to belong to 
$S^r_{p,q}B(\R)$. 
Again in view of the cross-norm property this implies that the product $f \, \cdot \, g$ has to belong to $B^r_{p,q}(\re)$, which means that 
$B^r_{p,q}(\re)$ itself has to be an algebra.
But in this case it is well-known that the given restrictions are necessary and sufficient, we refer, e.g., 
  to \cite{Tr1}, \cite{Tr83} and \cite{RS}.
The proof is complete.\qed


\subsection{Proofs of Lemma \ref{glatt} and Lemma \ref{uniform}}\label{proofs-Teil1}


We  recall some results  about the dual spaces of $S^r_{p,q}B(\R)$. 
For $1\le p \le \infty$ the conjugate exponent $p'$ is determined by $\frac{1}{p}+\frac{1}{p'}=1$.  
It will be convenient to work with the closure of $\cs(\R)$ in these  spaces.

\begin{definition}
By $\mathring{S}^r_{p,q}B(\R)$ we denote the closure of $\cs(\R)$ in $S^r_{p,q}B(\R)$.  
\end{definition}

As in the isotropic case we have
\[
\mathring{S}^r_{p,q}B(\R)  = {S}^r_{p,q}B(\R) \qquad \Longleftrightarrow \qquad \max (p,q) < \infty\, .
\]
Because of the density of $\cs(\R)$ in these spaces any element of the dual space can be interpreted as an element of $\cs' (\R)$.
Hence, a distribution  $f\in \cs'(\R)$ belongs to the dual space $(\mathring{S}^r_{p,q}B(\R))'$  
 if and only if there exists a positive constant $c$ such that
\beqq
|f(\varphi)|\leq c \, \|\varphi|S^r_{p,q}B(\R)\|  \qquad \text{holds for all}\ \  \varphi \in \cs(\R).
\eeqq

\begin{proposition}\label{dual2}
Let $r\in \re$. If $1\leq p \le \infty$ and $1 \le q\le \infty$, then it holds
\[
 [\mathring{S}^r_{p,q}B(\R)]'=S^{-r}_{p',q'}B(\R).
\]
\end{proposition}

We refer to Hansen \cite{Hansen} and \cite{KS15} for most of the  details.
In case $p=\infty$ we refer to Triebel \cite[2.5.1]{Tr78}, in particular to Remark 7 there,  where the isotropic case is treated.
Essentially the arguments used in the isotropic case carry over to the dominating mixed case.
We omit details.  
\vskip 3mm
Now we are in position to prove  Lemma \ref{glatt}.
\vskip 3mm
\noindent
{\bf Proof of Lemma \ref{glatt}}. Theorem \ref{main-be} yields that $C_0^\infty (\R)$ is a subset of 
$M(S^r_{p,q} B(\R))$, if $1\le p,q \le \infty $ and $r>1/p$.
Hence, it will be enough to deal with $r \le 1/p$.
\\
{\em Step 1.} Let $0 < r \le 1/p$.
Therefore we proceed as in proof of Theorem \ref{main-be}. Let $g \in C_0^\infty (\R)$ and $f \in S^r_{p,q}B(\R)$. 
Again we distinguish into the cases $e= \emptyset$ and  $e \neq \emptyset$. Concerning the first one we may argue as above.
Concerning the second one, we notice that we have to estimate again the quantities $S_{e,u}$, see 
\eqref{ws-117}. 
\\
{\em Substep 1.1.}  Let $u_i \le m$ for all $i \in e$.
Clearly 
\beqq
 \sup_{|h_i| < 2^{-k_i}, i \in e}  && \hspace*{-0.8cm}
\big\| \Delta_{h}^{2\bar{m} -u ,e}g(\cdot+u  \diamond h) \Delta_{h}^{u ,e}f(\cdot)|L_p(\R)\big\|  
\\
& \lesssim &  \Big(\prod_{i \in e} 2^{-k_i(2m-u_i)}\Big)
\sup_{|h_i| < 2^{-k_i}, i \in u}
\big\| \Delta_{h}^{u ,e}f\, |L_p(\R)\big\|
  \lesssim   \Big(\prod_{i \in e} 2^{-k_i m}\Big)
\big\| \, f\, |L_p(\R)\big\|\, .
\eeqq
Inserting this into the definition of the $S_{e,u}$, we find
\beqq
\Bigg\{\sum\limits_{k \in \N_0^d(e)} 2^{r|k|_1 q} && \hspace{-0.7cm} \bigg(\sup_{|h_i| < 2^{-k_i}, i \in e} 
\big\| \Delta_{h}^{2 \bar{m} - u ,e}g(\cdot+ u  \diamond h)\Delta_{h}^{ u ,e}f(\cdot)|L_p(\R)\big\|\bigg)^q\Bigg\}^{1/q}
\\
&\le &
\Bigg\{\sum\limits_{k \in \N_0^d(e)} 2^{r|k|_1 q} \Big(\prod_{i \in e} 2^{-k_i m}\Big)^q
\Bigg\}^{1/q} \, \big\| \, f \, |L_p(\R)\big\|
\\
&\lesssim & \| \, f \, | S^r_{p,q}B(\R)\|\, ,
\eeqq
since $m>r$.
\\
{\em Substep 1.2.} The case $m\leq u_i\leq 2m$ for all $i\in e$ is treated as  Step 2 in the proof of Theorem \ref{main-be}.\\
{\em Substep 1.3.} The remaining cases.  Let $e=\{1, \ldots ,N\}$ for some natural number $N$, $N \le d$. 
In addition we assume 
\[
u=(u_{1}, \ldots ,u_L,u_{L+1}, \ldots ,u_N,0,\ldots,0)
\] 
with
$$ m\leq u_i\leq 2m,\quad i=1, \ldots ,L,\qquad 0\leq u_i<m, \quad i= L+1, \ldots ,N$$
and $1\leq L\le N$ and $L< d$. 
For brevity we put
\[
e_1:= \{L+1,\ldots,N\} \qquad \mbox{ and } \qquad  e_2:=\{1,\ldots,L \}\, . 
\]
By assumption both sets are nontrivial.
Each $k \in \N_0^d (e)$ can be written as a sum $k =k^1 + k^2 $, $k^1 \in \N_0^d (e_1)$, $k^2 \in \N_0^d (e_2)$.
Inserting this into the definition of the $S_{e,u}$, we find
\beqq
&& \hspace{-0.7cm}
\Bigg\{\sum\limits_{k \in \N_0^d(e)} 2^{r|k|_1 q}  \bigg(\sup_{|h_i| < 2^{-k_i}, i \in e} 
\big\| \Delta_{h}^{2 \bar{m} - u ,e}g(\cdot+ u  \diamond h)\Delta_{h}^{ u ,e}f(\cdot)|L_p(\R)\big\|\bigg)^q\Bigg\}^{1/q}
\\
&\le &
\Bigg\{\sum\limits_{k^1 \in \N_0^d(e_1)} \sum\limits_{k^2 \in \N_0^d(e_2)} 2^{r|k|_1 q} \Big(\prod_{i \in e_2} 2^{-k_i(2m-u_i)}\Big)^q
\bigg(\sup_{|h_i| < 2^{-k_i}, i \in e_1} 
\big\| \Delta_{h}^{ u ,e_1}f \, |L_p(\R)\big\|\bigg)^q\Bigg\}^{1/q}
\\
& = &
\Bigg\{\sum\limits_{k^2 \in \N_0^d(e_2)} 2^{r|k^2|_1 q} \Big(\prod_{i \in e_2} 2^{-k_im}\Big)^q 
\sum\limits_{k^1 \in \N_0^d(e_1)} 2^{r|k^1|_1 q}
\bigg(\sup_{|h_i| < 2^{-k_i}, i \in e_1} \big\| \Delta_{h}^{ u ,e_1}f \, |L_p(\R)\big\|\bigg)^q \Bigg\}^{1/q} 
\\
& \lesssim &
\, \| \, f \, | S^r_{p,q}B(\R)\|\, . 
\eeqq
This proves the claim in case $r>0$ (we do not need $r\le 1/p$).
\\
{\em Step 2.} Let $r<0$. We shall argue by duality. 
Observe that the adjoint operator to $T_g$ is given by $T_{\bar{g}}$ and 
$g \in M(S^r_{p,q}B(\R))$ if and only if $\bar{g} \in M(S^r_{p,q}B(\R))$.
Hence, if $T_g \in \cl (S^r_{p,q}B(\R))$, then  $T_g \in \cl ((S^{r}_{p,q}B(\R))')$ follows.
\\
{\em Substep 2.1.}
Let $\max (p,q)< \infty$. Then Proposition \ref{dual2} and Step 1 yield
\be\label{ws-118}
C_0^\infty (\R) \subset M(S^{-r}_{p',q'}B(\R))\, .
\ee
{\em Substep 2.2.} Let $\max (p,q) =  \infty$. Then $\mathring{S}^r_{p,q} B(\R)$ is a proper subspace 
of ${S}^r_{p,q} B(\R)$. 
If $ g\in C_0^\infty (\R)$ and $g \in M({S}^r_{p,q} B(\R))$ then 
$g \in M(\mathring{S}^r_{p,q} B(\R))$ as well.
The same duality argument as in Substep 2.1 leads to \eqref{ws-118} also in this case.
Hence, \eqref{ws-118} is valid for all $1 \le p,q \le \infty$ and all $r>0$.
\\
{\em Step 3.} The case $r=0$.
We proceed by complex interpolation.
Let $X$ be a quasi-Banach space  of distributions. 
By $\accentset{\diamond}{X}$ we denote the closure in $X$ of the set of all infinitely
differentiable functions $g$ such that $D^{{\alpha}} g \in X$ for all ${\alpha} \in \N_0^d$.

\begin{proposition}\label{inter2}
Let $\Theta\in (0,1)$, $r_i\in \mathbb{R}$ and  $1\le p_i,q_i\leq \infty$, $i=1,2$. 
\\
{\rm (i)} Suppose 
\[
\min \Big(\max(p_1,q_1), \max(p_2,q_2)\Big)<\infty \, .
\]
If $r_0,p_0$ and $q_0$ are given by
\be\label{ws-120}
\frac{1}{p_0}=\frac{1-\Theta}{p_1}+\frac{\Theta}{p_2},
\qquad \frac{1}{q_0}=\frac{1-\Theta}{q_1}+\frac{\Theta}{q_2},\qquad r_0=(1-\Theta)r_1+\Theta r_2,
\ee
then
\beqq
S^{r_0}_{p_0,q_0}B(\R)=[S^{r_1}_{p_1,q_1}B(\R),\, S^{r_2}_{p_2,q_2} B(\R)]_{\Theta}.
\eeqq
{\rm (ii)} Let $r_1 \neq r_2$.
If $r_0$ and $q_0$ are defined  as in \eqref{ws-120}, then
\[
\accentset{\diamond}{S}^{r_0}_{\infty,q_0}B(\R)=[S^{r_1}_{\infty,q_1}B(\R),\, S^{r_2}_{\infty,q_2} B(\R)]_{\Theta},
\]
{\rm (iii)} Let $ 1 \le p_1 = p_2<\infty$ and $r_1 \neq r_2$. Let  $r_0,p_0$ and $q_0$ be given by
\eqref{ws-120}, 
 then
\[
\accentset{\diamond}{S}^{r_0}_{p_0,\infty}B(\R)=[S^{r_1}_{p_1,\infty}B(\R),\, S^{r_2}_{p_2,\infty} B(\R)]_{\Theta},
\]
\end{proposition}

We refer to Vybiral \cite[Theorem 4.6]{Vybiral}  concerning part (i). 
The isotropic counterparts of parts (ii), (iii) may be found in  Yuan, S., Yang \cite[pp. 1857/1858]{ysy}.
The arguments carry over to the dominating mixed case.
\\
{\em Substep 3.1.} Let $1 \le p,q < \infty $.
Combining Step 1, Step 2, Proposition \ref{inter2}(i) and the interpolation property of the complex method we conclude that 
\[
C_0^\infty (\R) \subset M(S^{0}_{p,q}B(\R))\, .
\]
{\em Substep 3.2.} Let $1 < p \le \infty $ and $q=\infty$.
We argue by duality as in Step 2.
$C_0^\infty (\R) \subset M(S^{0}_{p',1}B(\R))$ yields
$C_0^\infty (\R) \subset M(S^{0}_{p,\infty}B(\R))$.
\\
{\em Substep 3.3.} Let $p=1$ and $q=\infty$. Again we use duality in combination with 
\[
 C_0^\infty (\R) \subset M(\mathring{S}^{0}_{\infty,1}B(\R))\, .
\]
The proof is complete. \qed

\begin{remark}
 \rm 
A closer look to the proof yields
\[
 S^{t}_{\infty,\infty}B(\R) \hookrightarrow M(S^{r}_{p,q}B(\R))
\]
if  $t>|r|$. This follows from the characterization of $S^{t}_{\infty,\infty}B(\R)$ by differences, see Proposition \ref{diff}.
\end{remark}

{~}\\
{\bf Proof of Lemma \ref{uniform}}.
Since $S^{r}_{p,q}B(\R)$ is translation invariant the associated multiplier space has this property as well.
Because of $\psi_\mu \in C_0^\infty (\R)$ Lemma \ref{glatt} yields 
$\psi_\mu \, \cdot \, f\in S^{r}_{p,q}B(\R)$ for all $f\in S^{r}_{p,q}B(\R)$.
Consequently
\beqq
\| \, \psi_\mu \, \cdot \, f\, |S^{r}_{p,q}B(\R)\| &  = &  \| \, \psi \, \cdot \, f (\cdot + \mu)\, |S^{r}_{p,q}B(\R)\|
\le c_\psi \, \| \, f (\cdot + \mu)\, |S^{r}_{p,q}B(\R)\|
\\
&=& c_\psi \, \| \, f \, |S^{r}_{p,q}B(\R)\|\, .
\eeqq
This proves the claim. \qed


\subsection{Proof of the characterization of the multiplier space}


First, we recall the following two results. The first one deals with traces on hyperplanes.

\begin{proposition}
\label{trace} Let $1\le p,q\leq \infty$ and $r>1/p$. 
Let further $L\in \N$ and $ L\leq d$. 
If    $f \in S^r_{p,q}B(\R)$ then the function 
\[
g(x_1,\ldots,x_L):=f(x_1,\ldots,x_L,x_{L+1},\ldots,x_d)
\] of the $L$ variables  $x_1,\ldots, x_L$ ($x_{L+1},\ldots,x_d$ are considered as fixed) belongs to the space $ S^r_{p,q}B(\re^L)$.
\end{proposition}

\begin{proof}
For a proof we refer to \cite[Theorem 2.4.2]{ST}.
\end{proof}

Next we recall the  localization property of the spaces $S^r_{p,p}B(\R)$, proved in \cite{KS16}.

\begin{proposition}\label{local}
 Let $1\le p\leq \infty$ and $r>1/p$. Let further $\psi_{\mu}$, $\mu\in\Z$, be the functions defined in \eqref{ws-10}. Then we have
\beqq
\|f|S^r_{p,p}B(\R)\|\asymp \Big(\sum_{\mu\in \Z} \| \psi_{\mu}f|S^r_{p,p}B(\R)\|^p\Big)^{1/p}\,
\eeqq 
holds for all $f\in S^r_{p,p}B(\R)$.
\end{proposition}

The heart of the matter consists in the following proposition.

\begin{proposition}\label{q=inf} 
Let $1\le p\leq q\leq \infty$ and $r>1/p$. Then there exists a constant $C$ such that
\beqq
     \|\, f\, \cdot \, g \, |\,S^r_{p,q}B(\R)\|  \leq C\, \| \, f\, |S^r_{p,q}B(\R) \| \,   \|\,   g\, |S^r_{p,q}B(\R)_{\unif} \|\,
\eeqq
holds for all $f\in S^r_{p,q}B(\R)$ and $g\in S^r_{p,q}B(\R)_{\unif} $.
\end{proposition}

\begin{proof}
We follow the proof of Theorem \ref{main-be}. Again we make use of the characterizations by differences. 
Let $r<m\leq r+1$. Then we shall prove that  
\beqq
     \|\, f \, \cdot \,g\, |\,S^r_{p,q}B(\R)\|_{(2m)} \leq C\, \|\,  f \, |S^r_{p,q}B(\R) \| \,  \|\,  g\, |S^r_{p,q}B(\R)_{\unif} \|\,
\eeqq
holds for all $f \in S^r_{p,q}B(\R)$ and $g\in S^r_{p,q}B(\R)_{\unif} $.
\\
{\it Step 1.} Let $\psi $ be the function in Definition \ref{def-unif} and $\phi\in C_0^{\infty}(\R)$ chosen such that 
$\phi\equiv 1 $ on the support of $\psi$. It follows that
\beqq
 \|\, f\, \cdot \, g \, |\,S^r_{p,q}B(\R)\|_{(2m)}=  \Big\|\sum_{\mu\in \Z}\phi_{\mu}f\,  \psi_{\mu}g \,\Big|\,S^r_{p,q}B(\R)\Big\|_{(2m)} .
\eeqq
In case $1\le p < \infty $ the series $\sum_{\mu\in \Z}\phi_{\mu}f\,  \psi_{\mu}g$ is convergent in $S^r_{p,q}B(\R)$, in case $p=\infty$
we use the fact, that the sum is locally finite.
Clearly
\beqq 
\Big\| \sum_{\mu\in \Z}\phi_{\mu}f \psi_{\mu}g\Big|L_p(\R)\Big\| \, 
&\leq & \, \Big\| \sum_{\mu\in \Z}|\phi_{\mu}f| \cdot \| \psi_{\mu}g|C(\R)\|\Big|L_p(\R)\Big\| 
\\
 &\lesssim & \, \| f|L_p(\R)\|\,  \sup_{\mu\in \Z}\, \|\psi_{\mu}g|C(\R)\| 
\\ 
& \lesssim & \| \, f\, |S^r_{p,q}B(\R) \| \, \| \,  g \, |S^r_{p,q}B(\R)_\unif \|\, , 
\eeqq 
where we used Lemma \ref{emb1} in the last step.
For $e\subset [d]$,  $e\not=\emptyset$, we have
\beqq
\Delta_{h}^{2\bar{m},e}\bigg(\sum_{\mu\in \Z}\phi_{\mu}f\, \cdot \, \psi_{\mu}g\bigg)(x)
=\sum_{ |u|_\infty\leq 2m}\sum_{\mu\in \Z} C_{ u }\Delta_{h}^{2 \bar{m} - u ,e}(\phi_{\mu}f)(x+ u  \diamond h)
\, \Delta_{h}^{ u ,e}(\psi_{\mu}g)(x) \,, 
\eeqq
$h\in \R$, 
where $2\bar{m}-u=(2m-u_1, \ldots ,2m-u_d)$, see \eqref{mot}. This makes clear that we have to estimate  the terms
\be\label{hai-2}
S_{e,u}:=\Bigg\{ \sum_{k\in \N_0^d(e)} 2^{r|k|_1q }\sup_{|h_i| < 2^{-k_i}, i \in e} 
\Big\| \sum_{\mu\in \Z}\big|  \Delta_{h}^{2 \bar{m} - u ,e}(\phi_{\mu}f)(\cdot+ u  \diamond h)\, 
\Delta_{h}^{ u ,e}(\psi_{\mu}g)(\cdot)\big|\Big|L_p(\R)\Big\|^q\Bigg\}^{1/q} .
\ee 
For brevity  we put 
\[
P_k: = \Big\| \sum_{\mu\in \Z}\big|  \Delta_{h}^{2 \bar{m} - u ,e}(\phi_{\mu}f)(\cdot+ u  \diamond h)\, 
\Delta_{h}^{ u ,e}(\psi_{\mu}g)(\cdot) \big|\, \, \Big|L_p(\R)\Big\| \, .
\]
{\it Step 2.} Estimate of $S_{e,u}$ in case $u_i<m$ for all $i\in e$. 
We have
\be\label{u<m}
P_k
 \lesssim  \Big\| \sum_{\mu\in \Z}   
 | \Delta_{h}^{2\bar{m} - u ,e}(\phi_{\mu}f)(x) |\Big|L_p(\R)\Big \|\cdot 
 \sup_{\mu\in \Z}\sup_{x\in \R} |\Delta_{h}^{u ,e}(\psi_{\mu}g)(x)| .
\ee
By Lemma \ref{emb1} it is easily seen that 
\beqq
\sup_{\mu\in \Z}\sup_{x\in \R} |\Delta_{h}^{u ,e}(\psi_{\mu}g)(x)| & \lesssim & \sup_{\mu\in \Z}\sup_{x\in \R} | (\psi_{\mu}g)(x)| 
\\
&=& \sup_{\mu\in \Z} \| \psi_{\mu}g |C(\R)\| \lesssim \|  g|S^r_{p,q}B(\R)_\unif \|.
\eeqq
We estimate the first term on the right-hand side of \eqref{u<m} by using the decomposition
\[
\phi_{\mu}f=\phi_{\mu}\sum_{ \ell \in \Z}  \gf^{-1}  \varphi_{k+ \ell } \gf f=\sum_{ \ell \in \Z}  \phi_{\mu}f_{k+\ell}, 
\] 
see Substep 3.1 in the proof of Theorem \ref{main-be}. It follows that
\[
 \Big\| \sum_{\mu\in \Z}   | \Delta_{h}^{2\bar{m} - u ,e}(\phi_{\mu}f)(\cdot) |\Big|L_p(\R)\Big \|
 \leq  \sum_{\ell \in \Z} \Big\| \sum_{\mu\in \Z}   | \Delta_{h}^{2\bar{m} - u ,e}(\phi_{\mu}f_{k+\ell})(\cdot) |\Big|L_p(\R)\Big \| . 
\] 
Again we shall use the notation
\[
 \omega (\ell):= \big\{i \in \{1, \ldots , d\}:~ \ell_i <0 \big\} \qquad \mbox{and}\qquad 
 \overline{\omega} (\ell):= \big\{i \in \{1, \ldots , d\}:~ \ell_i \ge 0 \big\}\, .
\] 
Note that there exists a positive constant $c$ such that  $|x-\mu|> c$ implies $| \Delta_{h}^{2\bar{m} - u ,e}(\phi_{\mu}f_{k+\ell})(x) | \equiv 0$ for all $\mu$.
In case  $|x-\mu|\leq c$ Lemma \ref{ddim-2} yields 
\beqq
| \Delta_{h}^{2\bar{m} - u ,e}(\phi_{\mu}f_{k+\ell})(x) |&\lesssim & 
\bigg(\prod_{\omega(\ell)\cap e}2^{ \ell_i(2m-u_i)} \prod_{\bar{\omega}(\ell)\cap e}2^{ \ell_ia}\bigg)P_{2^{k+ \ell },a}f_{k+ \ell }(x)\\
&\leq & 
\bigg(\prod_{\omega(\ell)\cap e}2^{ \ell_i m} \prod_{\bar{\omega}(\ell)\cap e}2^{ \ell_ia}\bigg)P_{2^{k+ \ell },a}f_{k+ \ell }(x),
\eeqq
since $\ell_i<0$ and $u_i\leq m$ for  $i\in \omega(\ell)\cap e$. We choose $a$ such that $1/p<a<r$.  Hence
\beqq
\Big\| \sum_{\mu\in \Z}   | \Delta_{h}^{2\bar{m} - u ,e}(\phi_{\mu}f)(\cdot) |\Big|L_p(\R)\Big \| 
& \lesssim & 
\sum_{\ell\in \Z}\bigg(\prod_{\omega(\ell)\cap e}2^{ \ell_i m} \prod_{\bar{\omega}(\ell)\cap e}2^{ \ell_ia}\bigg) \|P_{2^{k+ \ell },a}f_{k+ \ell }|L_p(\R)\|  
\\
& \lesssim & \sum_{\ell\in \Z}\bigg(\prod_{\omega(\ell)\cap e}2^{ \ell_i m} \prod_{\bar{\omega}(\ell)\cap e}2^{ \ell_ia}\bigg)\| f_{ k+ \ell }|L_p(\R)\|\,.
\eeqq 
see Theorem \ref{peetremax}. This implies 
\beqq
P_k  \lesssim 
\sum_{\ell\in \Z}\bigg(\prod_{\omega(\ell)\cap e}2^{ \ell_i m} \prod_{\bar{\omega}(\ell)\cap e}2^{ \ell_ia}\bigg)\| f_{ k+ \ell }|L_p(\R)\| 
\cdot  \|  g|S^r_{p,q}B(\R)_\unif \|. 
\eeqq
Inserting this into \eqref{hai-2},  we obtain 
\beqq
S_{e,u} &\lesssim &  
\Bigg\{ \sum_{k\in \N_0^d(e)}\bigg[  2^{r|k|_1 } \sum_{\ell\in \Z}\bigg(\prod_{\omega(\ell)\cap e}2^{ \ell_i m} \prod_{\bar{\omega}(\ell)\cap e}2^{ \ell_ia}\bigg)
\| f_{ k+ \ell }|L_p(\R)\|\bigg]^q \Bigg\}^{1/q}  
\|  g|S^r_{p,q}B(\R)_\unif \|  
\\
&\lesssim & \sum_{\ell\in \Z}\Bigg\{\sum_{k\in \N_0^d(e)}\bigg(2^{r|k|_1} \prod_{i\in \omega(\ell)\cap e}2^{ \ell_i m} 
\prod_{i\in \bar{\omega}(\ell)\cap e}2^{ \ell_ia}  \bigg)^q \| f_{ k+ \ell }|L_p(\R)\|^q\Bigg\}^{1/q}
\|  g|S^r_{p,q}B(\R)_\unif \| \, .
\eeqq
Observe that
\beq
2^{-r|k+\ell|_1}\bigg(2^{r|k|_1} \prod_{i\in \omega(\ell)\cap e}2^{ \ell_i m} \prod_{i\in \bar{\omega}(\ell)\cap e}2^{ \ell_ia} \bigg)
&= & \prod_{i\in \omega(\ell)\cap e}2^{\ell_i(m-r)}\prod_{i\in \bar{\omega}(\ell)\cap e}2^{\ell_i(a-r)}\prod_{i\in [d]\backslash e}2^{-\ell_ir} \nonumber
\\
&\leq & \prod_{i=1}^d2^{-|\ell_i|\delta}, \label{key-0}
\eeq
where $\delta :=\min(m-r,r-a,r)>0$. This leads to
\beqq
S_{e,u}&\lesssim &\sum_{\ell\in \Z} \Big( \prod_{i=1}^d2^{-|\ell_i|\delta} \Big)\cdot \| f|S^r_{p,q}B(\R)\| \cdot  \|  g|S^r_{p,q}B(\R)_\unif \| 
\\
&\lesssim & \| f|S^r_{p,q}B(\R)\| \cdot  \|  g|S^r_{p,q}B(\R)_\unif \|.
\eeqq
{\it Step 3.} Estimate of $S_{e,u}$ in case $u_i\geq m$ for all $ i\in e$. 
We have
\beqq
   \big\| \Delta_h^{2\bar{m} - u,e}(\phi_{\mu}f)(\cdot+uh) && \hspace{-0.7cm} \Delta_h^{u,e}(\psi_{\mu}g)(\cdot) \big|  L_p(\R)\big\|^p \\ 
   & \leq &
     \big\| \Delta_h^{2\bar{m} - u,e}(\phi_{\mu}f)(\cdot) \big|C(\R) \big\|^p \cdot  \big\|\Delta_h^{u,e}(\psi_{\mu}g)(\cdot) \big|  L_p(\R)\big\|^p   
 \\
 & \lesssim &  \big\| \phi_{\mu}f  \big|C(\R) \big\|^p \cdot  \big\|\Delta_h^{u,e}(\psi_{\mu}g)(\cdot) \big|  L_p(\R)\big\|^p   .
\eeqq
Inserting this into $S_{e,u}$ and applying  the triangle inequality with $q/p\geq 1$ we have found
\beqq
S_{e,u} &\lesssim & \Bigg\{\sum_{k\in \N_0^d(e)}\bigg(2^{r|k|_1p} \sup_{|h_i|< 2^{-k_i}, i\in e}  \sum_{\mu\in \Z}
\big\| \phi_{\mu}f \big|C(\R) \big\|^p \cdot  \big\|\Delta_h^u(\psi_{\mu}g)(\cdot) \big|  L_p(\R)\big\|^p\bigg)^{q/p} \Bigg\}^{1/q}  
\\
&\leq & \bigg\{\sum_{\mu\in \Z}   \big\| \phi_{\mu}f \big|C(\R) \big\|^p  \cdot \big\| \psi_{\mu}g|S^r_{p,q}B(\R)\big\|^p_{(m)} \bigg\}^{1/p} 
\\
& \leq & \bigg\{\sum_{\mu\in \Z}   \big\| \phi_{\mu}f \big|C(\R) \big\|^p  \bigg\}^{1/p}  \cdot \big\|  g|S^r_{p,q}B(\R)_{\unif}\big\|.
\eeqq
Since $r>1/p$, there exists some $\varepsilon>0$ such that $r-\varepsilon>1/p$. 
This implies $S^{r-\varepsilon}_{p,p}B(\R)\hookrightarrow C(\R)$, see Lemma \ref{emb1}. Hence, by means of the localization property 
of $S^{r-\varepsilon}_{p,p}B(\R)$, see Proposition \ref{local}, 
\beqq
\bigg( \sum_{\mu\in \Z} \big\| \phi_{\mu}f \big|C(\R) \big\|^p \bigg)^{1/p} 
\lesssim  \bigg( \sum_{\mu\in \Z} \big\| \phi_{\mu}f \big|S^{r-\varepsilon}_{p,p}B(\R) \big\|^p    \bigg)^{1/p} \asymp \| f|S^{r-\varepsilon}_{p,p}B(\R)\|\, .
\eeqq
Now the elementary  embedding $S^r_{p,q}B(\R) \hookrightarrow S^{r-\varepsilon}_{p,p}B(\R)$ implies 
 \[
S_{e,u}  \lesssim  \| \, f\, | S^r_{p,q }B(\R)\| \, \|\,   g \, |S^r_{p,q}B(\R)_\unif \|.
\]
{\it Step 4.} Estimate of $S_{e,u}$  for the remaining cases. We shall use the same notation as in proof of Theorem  \ref{main-be}, 
Step 3, i.e., we assume that  $e=\{1, \ldots ,N\}$ for some natural number $N$, $N \le d$, 
\[
u=(u_{1}, \ldots ,u_L,u_{L+1}, \ldots ,u_N,0,\ldots,0)
\] 
with 
\[ 
m\leq u_i\leq 2m,\quad i=1, \ldots ,L,\qquad 0\leq u_i<m, \quad i= L+1, \ldots ,N
\]
and $1\leq L\le N$ and $L< d$. 
Again we define 
\[
e_1:= \{L+1,\ldots,N\} \qquad \mbox{ and } \qquad  e_2:=\{1,\ldots,L \}\, . 
\]
Both sets are nontrivial. This covers all remaining cases up to an enumeration. 
Again we make use of  $\N_0^d(e)=\N_0^d(e_1)\cup \N_0^d(e_2)$. For brevity we put 
\[
 T_{\mu,h} (f,g):= \| \Delta_{h}^{2 \bar{m} - u ,e}(\phi_{\mu}f)(\cdot+ u  \diamond h)\Delta_{h}^{ u ,e}(\psi_{\mu}g)(\cdot) |L_p(\R)\|\, .
\]
Then, because of  $q/p\geq 1$,  \eqref{hai-2} yields
\beq\label{ws-107}
S_{e,u} &\le &  \Bigg\{\sum_{k\in \N_0^d(e)}  2^{r|k|_1q }\sup_{|h_i| < 2^{-k_i}, i \in e} 
\bigg( \sum_{\mu\in \Z} T_{\mu,h}(f,g)^p\bigg)^{q/p}\Bigg\}^{1/q} 
\nonumber
\\
& \leq &  \Bigg\{\sum_{k^1\in \N_0^d(e_1)}2^{r|k^1|_1q }
\Bigg[\bigg(  \sum_{k^2\in \N_0^d( e_2)} 2^{r|k^2|_1q }\bigg[ \sum_{\mu\in \Z}\sup_{|h_i| < 2^{-k_i}, i \in e}  T_{\mu,h}(f,g)^p
\bigg]^{q/p}\bigg)^{p/q}\Bigg]^{q/p}\Bigg\}^{1/q} 
\nonumber
\\
& \le & \Bigg\{\sum_{k^1\in \N_0^d(e_1)}2^{r|k^1|_1q }
\Bigg[   \sum_{\mu\in \Z} \bigg( \sum_{k^2\in \N_0^d(e_2)}2^{r|k^2|_1q }\sup_{|h_i| < 2^{-k_i}, i \in e}  T_{\mu,h}(f,g)^q\bigg)^{p/q}\Bigg]^{q/p}\Bigg\}^{1/q} .
\eeq
We  consider the integral
\beqq
 T_{\mu,h}(f,g)^p
 & \leq & \bigg(\int\limits_{\re^{d-L}}\sup_{ x_i\in \re\atop i\leq L} \big|\Delta_{h}^{2\bar{m} - u ,e}(\phi_{\mu}f )(x+ u  \diamond h) \big|^p\prod_{i=L+1}^d d  x_i \bigg)     \bigg(\int\limits_{\re^{L}}\sup_{x_i\in \re\atop L<i\leq d} \big|\Delta_{h}^{ u ,e }(\psi_{\mu}g)(x) \big|^p\prod_{i=1}^L d  x_i \bigg)  
 \\
  & \lesssim & \bigg(\int\limits_{\re^{d-L}}\sup_{ x_i\in \re\atop i\leq L} \big|\Delta_{h}^{\bar{m} ,e_1}(\phi_{\mu}f )(x ) \big|^p\prod_{i=L+1}^d d  x_i \bigg)      \bigg(\int\limits_{\re^{L}}\sup_{x_i\in \re\atop L<i\leq d} \big|\Delta_{h}^{\bar{m} ,e_2 }(\psi_{\mu}g)(x) \big|^p\prod_{i=1}^L d  x_i \bigg).  
\eeqq
Let $G:~ \R \to \C $ be a given function and $a\subset [d]$. When we write 
$\| \, G \, |S^{t}_{p,p}B(\re^{a}) \big\|$, then we mean that the norm is taken with respect to the variables with indexes in $v$, the remaining are considered as frozen.
In addition we shall use the notation 
\[
 \Delta_{h,t}^{\bar{m},e_2\cup\, a} : = \Delta_{t}^{\bar{m},a}(\Delta_h^{\bar{m},e_2})\quad\text{and}\quad 
 \omega_{\bar{m}}^{e_2\cup a}(G,2^{-k^2},2^{-\ell})_p := \sup_{|h_i|<2^{-k_i},i\in e_2\atop |t_i|<2^{-\ell_i},i\in a}\big\|\Delta_{h,t}^{\bar{m},e_2\cup a}(G)\big|L_p(\R)\big\| .
\]
Since $r>1/p$, there exists $\varepsilon_1>0$ such that $r-\varepsilon_1>1/p$.  From Lemmas \ref{emb1}, \ref{red}, Proposition \ref{trace} and some monotonicity arguments we conclude
\beq  
\int\limits_{\re^{L}}&& \hspace{-0.8cm}\sup_{x_i\in \re\atop L<i\leq d}  \big|\Delta_{h}^{\bar{m} ,e_2 }(\psi_{\mu}g)(x) \big|^p\prod_{i=1}^L d  x_i   
\lesssim       \int\limits_{\re^{L}}  \big\|\Delta_{h}^{\bar{m} ,e_2 }(\psi_{\mu}g)(x)\big|S^{r-\varepsilon_1}_{p,p}B(\re^{e_1\cup e_0}) \big\|^p\prod_{i=1}^L d  x_i 
\nonumber    
\\
 & \lesssim &\int\limits_{\re^{L}}\Bigg(\sum_{a\subset \{L+1,...,d\}} \int\limits_{[-1,1]^{|a|}} \prod_{i \in a} |t_i|^{(-r+\varepsilon_1)p} \big\|  \Delta_{  h,t}^{\bar{m},e_2\cup a} (\psi_\mu g)(\cdot) \big|
   L_p (\re^{d-L})\big\|^p \prod_{i \in a} \frac{dt_i}{|t_i|} \Bigg) \prod_{i=1}^L d  x_i 
\nonumber
\\
 & = &  \sum_{a\subset \{L+1,...,d\}} \int\limits_{[-1,1]^{|a|}} \prod_{i \in a} |t_i|^{(-r+\varepsilon_1)p} \big\|  \Delta_{  h,t}^{\bar{m},e_2\cup a} (\psi_\mu g)(\cdot) \big|
 L_p (\R)\big\|^p \prod_{i \in a} \frac{dt_i}{|t_i|} 
 \nonumber \\
&\lesssim &  \sum_{a\subset \{L+1,...,d\}}   \sum_{\ell \in \N_0^d(a)}2^{|\ell|_1(r-\varepsilon_1)p} \omega_{\bar{m}}^{e_2\cup a}(\psi_\mu g,2^{-k^2},2^{-\ell})_p^p . \label{key-1}
\eeq
We need one more abbreviation
\[
F_\mu (k^1) :=  \sup_{|h_i| < 2^{-k_i}, i \in  e_1} \, \int\limits_{\re^{d-L}}\sup_{ x_i\in \re\atop i\leq L} 
\big|\Delta_{h}^{\bar{m} ,e_1}(\phi_{\mu}f )(x ) \big|^p\prod_{i=L+1}^d d  x_i \, .
\]
This leads to the estimate of the term in $[\ldots]$ in \eqref{ws-107}
\beqq
&&
\sum_{\mu\in \Z} \bigg( \sum_{k^2\in \N_0^d(e_2)}2^{r|k^2|_1q }\sup_{|h_i| < 2^{-k_i}, i \in e}  T_{\mu,h}(f,g)^q\bigg)^{p/q} \\
&\lesssim&
 \sum_{\mu\in \Z} F_\mu(k^1) \Bigg\{\sum_{k^2\in \N_0^d(e_2)}  2^{r|k^2|_1q }\bigg(   \sum_{a\subset \{L+1,...,d\}}   \sum_{\ell \in \N_0^d(a)}2^{|\ell|_1(r-\varepsilon_1)p} \omega_{\bar{m}}^{e_2\cup a}(\psi_\mu g,2^{-k^2},2^{-\ell})_p^p\bigg)^{q/p}\Bigg\}^{p/q}.
\eeqq
Next we apply the elementary inequality 
\be \label{key-2}
\sum_{j\in \N_0} |a_j| \leq c  \, \Big(\sum_{j\in \N_0} 2^{j\varepsilon t}|a_j|^{t}\Big)^{1/t}\, , 
\ee
valid for all $\varepsilon>0$ and all $t\ge 1$. This inequality, used with $t= q/p$, yields
\beqq
\sum_{\mu\in \Z} && \hspace{-0.7cm}\bigg( \sum_{k^2\in \N_0^d(e_2)}2^{r|k^2|_1q }\sup_{|h_i| < 2^{-k_i}, i \in e}  T_{\mu,h}(f,g)^q\bigg)^{p/q} \\
&\lesssim& \sum_{\mu\in \Z} F_\mu(k^1) \Bigg\{ \sum_{a\subset \{L+1,...,d\}}  \sum_{k^2\in \N_0^d(e_2)} \sum_{\ell \in \N_0^d(a)}  
2^{r|k^2|_1q }\,   2^{|\ell|_1rq}\, \omega_{\bar{m}}^{e_2\cup a}(\psi_\mu g,2^{-k^2},2^{-\ell})_p^q\Bigg\}^{p/q} \\
&\lesssim&  \sum_{\mu\in \Z} F_\mu(k^1)   \|\,  \psi_\mu g\, | S^r_{p,q}B(\R)\|^p
\eeqq
since $a $ and $e_2$ are disjoint.
This can be inserted into the estimate of $S_{u,e}$ to get
\beqq  
S_{u,e} &\lesssim & 
\Bigg\{\sum_{k^1\in \N_0^d(e_1)}2^{r|k^1|_1q }\bigg[   \sum_{\mu\in \Z} F_\mu (k^1)\,  \|\, \psi_{\mu}g\, |S^r_{p,q}B(\R)\|^p \bigg]^{q/p}\Bigg\}^{1/q} 
\\
& \le & \Bigg\{\sum_{k^1\in \N_0^d(e_1)}2^{r|k^1|_1q }
\bigg[   \sum_{\mu\in \Z} F_\mu (k^1)      \bigg]^{q/p}\Bigg\}^{1/q} \, \| \, g\, |S^r_{p,q}B(\R)_{\unif}\|.
\eeqq 
To finish the proof it will be sufficient to show that 
\beqq
S_{u,e}^*:= 
\Bigg\{\sum_{k^1\in \N_0^d(e_1)}2^{r|k^1|_1q }\bigg[   \sum_{\mu\in \Z} F_\mu (k^1)      \bigg]^{q/p}\Bigg\}^{1/q} \leq C \,  \|\,  f\, |S^r_{p,q}B(\R)\|
\eeqq
holds for some constant $C$ independent of $f$.
Similar to \eqref{key-1} we conclude
\beqq
&& F_\mu (k^1) 
\lesssim      \sum_{v\subset  [L]} \sum_{j\in \N_0^d(v)} 2^{|j|_1(r-\varepsilon_1)p}    
\sup_{|h_i|<2^{-k_i}, i\in e_1\atop |s_i|< 2^{-j_i}, i\in v}\|\Delta_{h,s}^{\bar{m},e_1\cup\, v}(\phi_{\mu}f )\, |L_p(\R)\big\|^p  \,.
\eeqq
Note that $v\cap e_1=\emptyset$. Again we have to decompose $\phi_{\mu}f$. But this time we only split $f$. This results in 
\[
\phi_{\mu}f = \phi_{\mu}\sum_{ \ell \in \Z}  \gf^{-1}  \varphi_{k^1+j+\ell } \gf f=\sum_{ \ell \in \Z}  \phi_{\mu}f_{k^1 +j+\ell} \, ,  
\]
where $j$ and $k^1$ are at our disposal. With $k^1 \in \N_0^d(e_1)$ and $j \in \N_0^d(v)$, as in \eqref{ell}, we can assume 
\be\label{ell-1}
[d]\backslash (e_1\cup  v)\subset \overline{\omega}(\ell)\qquad \text{and}\qquad \omega(\ell)\subset (e_1\cup v).
\ee 
Let $c>0$ be chosen such that 
\beqq
\Delta_{h,s}^{\bar{m},e_1\cup\, v}(\phi \, \cdot \, f )(x) = 0 \qquad \mbox{if} \qquad |x|_\infty \ge c\, . 
\eeqq
We put  $Q_{\mu}: =\{ x \in \R: ~|x-\mu|_\infty \leq c \}$. 
Because of \eqref{ell-1}, Lemma \ref{ddim-2} yields
\beqq
 |\Delta_{h,s}^{\bar{m},e_1\cup\, v}(\phi_{\mu}f_{k^1+j+\ell} )(x) | \leq \bigg(\prod_{i\in \omega(\ell)}2^{ \ell_i m} 
 \prod_{\bar{\omega}(\ell)\cap (e_1\cup v)}2^{ \ell_i a} \bigg)P_{2^{k^1+j+\ell },a}f_{k^1+j+\ell }(x)\, , \qquad x\in Q_{\mu}\, ,
\eeqq
for all $h$, $|h_i|< 2^{-k_i}$, $i \in e_1$ and for all $s$, $|s_i|< 2^{-j_i}$, $i \in v$. For those pairs $(h,s)$, applying the triangle 
inequality with respect  to $L_p (\R)$, it follows  
\beqq
\|\Delta_{h,s}^{\bar{m},e_1\cup\, v}(\phi_{\mu}f ) |L_p(\re^d)\big\|^p & \lesssim&   \bigg\|\sum_{\ell\in \Z} \bigg(\prod_{i\in \omega(\ell)}2^{ \ell_i m}     
\prod_{\bar{\omega}(\ell)\cap (e_1\cup v)}2^{ \ell_i a} \bigg)P_{2^{k^1+j+\ell },a}f_{k^1+j+\ell }\, \bigg|L_p(Q_{\mu})\bigg\|^p \\
&\lesssim& \bigg[\sum_{\ell\in \Z} 
\bigg(\prod_{i\in \omega(\ell)}2^{ \ell_i m}     \prod_{\bar{\omega}(\ell)\cap (e_1\cup v)}2^{ \ell_i a} \bigg)
\big\|P_{2^{k^1+j+\ell },a}f_{k^1+j+\ell }\, \big| L_p(Q_{\mu})\big\|\bigg]^p\,.
\eeqq
Consequently we find
\beqq
 F_\mu (k^1) 
 \lesssim      \sum_{v\subset  [L]} \sum_{j\in \N_0^d(v)} 2^{|j|_1 (r-\varepsilon_1)p}\bigg[\sum_{\ell\in \Z} 
\bigg(\prod_{i\in \omega(\ell)}2^{ \ell_i m}     \prod_{\bar{\omega}(\ell)\cap (e_1\cup v)}2^{ \ell_i a} \bigg)
\big\|P_{2^{k^1+j+\ell },a}f_{k^1+j+\ell }\, \big| L_p(Q_{\mu})\big\|\bigg]^p   \,.
\eeqq
The final overlap property of the $Q_\mu$ leads to
\beqq
&& \hspace*{-0.7cm}\bigg\{\sum_{\mu\in \Z}   F_\mu (k^1) \bigg\}^{1/p}
\lesssim   \sum_{v\subset [L]} \sum_{\ell\in \Z}\Bigg\{    \sum_{j\in \N_0^d(v)}\bigg( 2^{|j|_1 (r-\varepsilon_1)}\prod_{i\in \omega(\ell)}2^{ \ell_i m}     
\prod_{\bar{\omega}(\ell)\cap (e_1\cup v)}2^{ \ell_i a} \bigg)^p 
\\
&& \hspace*{5cm} \times \quad
\sum_{\mu\in \Z}
\big \| P_{2^{k^1+j+\ell },a}f_{k^1+j+\ell }\, \big |L_p(Q_{\mu})\big\|^p\Bigg\}^{1/p}
\\ 
&\lesssim &  \sum_{v\subset [L]} \sum_{\ell\in \Z}\Bigg\{    \sum_{j\in \N_0^d(v)} \bigg(2^{|j|_1 (r-\varepsilon_1) }\prod_{i\in \omega(\ell)}2^{ \ell_i m}     
\prod_{\bar{\omega}(\ell)\cap (e_1\cup v)}2^{ \ell_i a} \bigg)^p \big\| P_{2^{k^1+j+\ell },a}f_{k^1+j+\ell }\, \big|L_p(\R)\big\|^p\Bigg\}^{1/p}
\\  
&\lesssim &   \sum_{v\subset [L]} \sum_{\ell\in \Z}\Bigg\{    \sum_{j\in \N_0^d(v)} \bigg(2^{|j|_1 (r-\varepsilon_1) }\prod_{i\in \omega(\ell)}2^{ \ell_i m}     
\prod_{\bar{\omega}(\ell)\cap (e_1\cup v)}2^{ \ell_i a} \bigg)^p \big\|  f_{k^1+j+\ell }\, \big|L_p(\R)\big\|^p\Bigg\}^{1/p},
\eeqq
where in the last step we employed Theorem \ref{peetremax}. The triangle inequality in $\ell_q$ yields
\beqq
S_{u,e}^*
&\lesssim &   \Bigg\{\sum_{k^1\in \N_0^d(e_1)}\Bigg[ \sum_{v\subset [L]} \sum_{\ell\in \Z}\Bigg(    \sum_{j\in \N_0^d(v)}
\bigg(2^{|k^1|_1r}2^{|j|_1 (r-\varepsilon_1) }\prod_{i\in \omega(\ell)}2^{ \ell_i m}     \prod_{\bar{\omega}(\ell)\cap (e_1\cup v)}2^{ \ell_i a} \bigg)^p
\\
&& \hspace*{5cm} \times\quad \big\|  f_{k^1+j+\ell }\, \big|L_p(\R)\big\|^p\Bigg)^{1/p}      \Bigg]^q\Bigg\}^{1/q}  
\\
& \lesssim &
   \sum_{v\subset [L]}\sum_{\ell \in \Z}\Bigg\{\sum_{k^1\in \N_0^d(e_1)}  \Bigg[    \sum_{j\in \N_0^d(v)}  \bigg(2^{|k^1|_1r}2^{|j|_1 (r-\varepsilon_1) }
   \prod_{i\in \omega(\ell)}2^{ \ell_i m}     \prod_{\bar{\omega}(\ell)\cap (e_1\cup v)}2^{ \ell_i a} \bigg)^p
   \\
&& \hspace*{5cm} \times\quad \big\| f_{k^1+j+\ell }\, \big|L_p(\R)\big\|^p\Bigg]^{q/p}       \Bigg\}^{1/q}  .
\eeqq
Next we apply the inequality \eqref{key-2} with $\varepsilon_2>0$  and $t=q/p$ to yield
\beqq
S_{u,e}^*
&\lesssim &
   \sum_{v\subset [L]}\sum_{\ell \in \Z}\Bigg\{\sum_{k^1\in \N_0^d(e_1)}      \sum_{j\in \N_0^d(v)}  
   \bigg(2^{|k^1|_1r}2^{|j|_1 (r-\varepsilon_1+\varepsilon_2) }\prod_{i\in \omega(\ell)}2^{ \ell_i m}     \prod_{\bar{\omega}(\ell)\cap (e_1\cup v)}2^{ \ell_i a} \bigg)^q
   \\
& &  \hspace*{5cm} \times\quad \big\| f_{k^1+j+\ell }\, \big|L_p(\R)\big\|^q       \Bigg\}^{1/q} . 
\eeqq
Since $\varepsilon_2>0$ is arbitrary we can choose  $\varepsilon_2<\varepsilon_1<r-1/p$ to get
\beqq
S_{u,e}^*
\lesssim \sum_{v\subset [L]}\sum_{\ell \in \Z}\bigg(    \prod_{i\in \omega(\ell)}2^{ \ell_i m}     \prod_{\bar{\omega}(\ell)\cap (e_1\cup v)}2^{ \ell_i a}\prod_{i=1}^d 2^{-\ell_ir}\bigg)\cdot \| f|S^r_{p,q}B(\R)\|.
\eeqq
Let $\delta_2:=\min(m-r,r-a,r)>0$. Then, as in \eqref{key-0} (see \eqref{ell-1}), we conclude
\beqq
  \prod_{i\in \omega(\ell)}2^{ \ell_i m}     \prod_{\bar{\omega}(\ell)\cap (e_1\cup v)}2^{ \ell_i a}\prod_{i=1}^d 2^{-\ell_ir}
  &= & \prod_{i\in \omega(\ell)}2^{\ell_i(m-r)}\prod_{\bar{\omega}(\ell)\cap (e_1\cup v)}2^{ \ell_i (a-r)}\prod_{i\in [d]\backslash(e_1\cup v)}2^{-r\ell_i}
  \\
 & \leq& \prod_{i=1}^d2^{-|\ell_i|\delta_2}\, , 
\eeqq
which finally  implies $S_{u,e}^* \lesssim  \|\,  f\, |S^r_{p,q}B(\R)\|$ and hence
\beqq
S_{e,u}      \lesssim  \|\, f\, |S^r_{p,q}B(\R)\| \, \| \,  g\, |S^r_{p,q}B(\R)_\unif\|.
\eeqq
The proof is complete.
\end{proof} 


\noindent
{\bf Proof of Theorem \ref{mul-spaceb}}. Theorem \ref{mul-spaceb} is the direct consequence of 
Proposition \ref{q=inf} and Lemma \ref{uniform}.


\noindent
{\bf Proof of Theorem \ref{negative}.} We may employ the same counterexamples as in case $p=q$ which is treated in 
\cite{KS16}.

\noindent
{\bf Proof of Corollary \ref{mul-spacec}}.
The characterization of the multiplier space in \eqref{ws-127} is an immediate consequence 
of Theorem \ref{mul-spaceb} and the duality argument as used in proof of Lemma \ref{glatt}.
We omit details.
\qed


\subsection{Proof of the assertions in the local case}


\noindent
{\bf Proof of Theorem \ref{main-be-1}}.
The {\em if}-part is obvious. To prove the {\em only if}-part we 
apply the arguments from the proof of Theorem \ref{main-be}
and conclude that 
there exists a constant $C>0$ such that 
\[
\| f\cdot g\,|\, B^r_{p,q} (\re)\| \leq C\|\,f\,|B^r_{p,q} (\re)\| \,  \|\, g\,|B^r_{p,q} (\re)\|
\]
holds for all $f,g\in B^r_{p,q} (\re) \cap C^\infty (\re)$ satisfying $\supp f,g \subset (0,1)^d$.
As in the proof of Theorem 2.6.2/1 in Triebel \cite{Tr78}  we conclude that 
$\mathring{B}^r_{p,q} ([0,1])$ must be embedded into $C([0,1])$. 
Again this is known to be equivalent to the given restrictions, see \cite{SiTr}.
\qed

\noindent
{\bf Proof of Theorem \ref{mul-spacew}}. 
Sufficiency follows from Theorem \ref{main-be}. Necessity is implied by the fact that the function 
$g = 1$ on $[0,1]^d$ belongs to all spaces $S^r_{p,q}B([0,1]^d)$.
Hence, a function $f\in M(S^r_{p,q}B([0,1]^d))$ has to satisfy $f\, \cdot \, g \in S^r_{p,q}B([0,1]^d)$
for this $g$ and therefore $f\in S^r_{p,q}B([0,1]^d)$.
\qed

\noindent
{\bf Proof of Theorem \ref{negativec}}.
It is enough to observe that the used counterexamples in the proof of 
Theorem \ref{negative} have compact support.
\qed


\end{document}